\documentclass{article}
\usepackage{graphicx} 
\usepackage{amsfonts}
\usepackage{amsmath}
\usepackage{enumitem}
\usepackage{algorithm}
\usepackage{algorithmic}
\usepackage{amsthm} 

\newtheorem{remark}{Remark}[section] 
\newtheorem{theorem}{Theorem}[section] 
\newcommand{\boxedeq}[1]{\begin{equation}\boxed{#1}\end{equation}}
\usepackage[margin=1.5in, top=1.0in]{geometry} 
\allowdisplaybreaks[4]
\newlist{questions}{enumerate}{1}
\setlist[questions]{label=(Q\arabic*), leftmargin=*, align=left}
\newenvironment{keywords}
{
    \par\noindent\textbf{Key words: }
}
{
    \par
}

\newenvironment{AMS}
{
    \par\noindent\textbf{AMS subject classifications. }
}
{
    \par
}

\title{A Fully Parallelizable Loosely Coupled Scheme for Fluid-Poroelastic Structure Interaction Problems}

\author{Shihan Guo\thanks{School of Mathematical Sciences, East China Normal University, Shanghai 200241, China 
  (hnguosh@gmail.com, bill950204@126.com).}
\and Yizhong Sun\footnotemark[1]  
\and Yifan Wang\thanks{Corresponding author. Department of Mathematics and Statistics, Texas Tech University, Lubbock, TX 79409, USA 
(Yifan.Wang@ttu.edu).}
\and Xiaohe Yue\thanks{Corresponding author. School of Mathematical Sciences, East China Normal University, Shanghai 200241, China (yuexiaohe666@gmail.com)}
\and Haibiao Zheng\thanks{School of Mathematical Sciences, Ministry of Education Key Laboratory of Mathematics and Engineering Applications, Shanghai Key Laboratory of PMMP,  East China Normal University, Shanghai 200241, China (hbzheng@math.ecnu.edu.cn).}
}

\begin{document}
\date{}
\maketitle

\begin{abstract}
We investigate the fluid-poroelastic structure interaction problem in a moving domain, governed by Navier-Stokes-Biot (NSBiot) system. First, we propose a fully parallelizable, loosely coupled scheme to solve the coupled system. At each time step, the solution from the previous time step is used to approximate the coupling conditions at the interface, allowing the original coupled problem to be fully decoupled into seperate fluid and structure subproblems, which are solved in parallel. Since our approach utilizes a loosely coupled scheme, no sub-iterations are required at each time step. Next, we conduct the energy estimates of this splitting method for the linearized problem (Stokes-Biot system), which demonstrates that the scheme is unconditionally stable without any restriction of the time step size from the physical parameters. Furthermore, we illustrate the first-order accuracy in time through two benchmark problems. Finally, to demonstrate that the proposed method maintains its excellent stability properties also for the nonlinear NSBiot system, we present numerical results for both $2D$ and $3D$ NSBiot problems related to real-world physical applications.
\end{abstract}


\section{Introduction}

The Navier-Stokes-Biot (NSBiot) model represents a sophisticated fusion of fluid dynamics, as governed by the Navier-Stokes equations, with poroelasticity, as described by the Biot equations \cite{biot1941general, biot1955theory}. This integrated model is specifically designed to capture the complex interplay between fluid flow and the deformation of a poroelastic medium. By coupling these two domains, the NSBiot model effectively simulates scenarios where fluid dynamics influence the deformation within poroelastic structure, and conversely, where the deformation of the medium impacts fluid flow and filtration process. This model has gained considerable attention in recent years due to its versatility and applicability across various fields. In geophysics, for example, the NSBiot model is crucial for understanding phenomena such as groundwater flow and soil stability, providing insights into how fluid movement can alter the mechanical properties of geological formations \cite{Benjamin2014, doi:10.1137/S1064827503429363}. In biomedical engineering, the model plays an essential role in simulating blood flow through tissues and organs, offering valuable perspectives on how the interaction between fluid and tissue mechanics influences physiological processes \cite{BADIA20097986, fluids7070222, Martina2014}. 

Over the past decade, significant advancements have been made in both the theoretical and numerical aspects of the NSBiot model \cite{BADIA20097986}. These developments have enhanced our understanding of fluid-structure interactions in poroelastic media, contributing to more accurate and efficient simulations. Here, we only provide a partial list of relevant work. Theoretical research has addressed critical aspects such as the existence of weak solutions in fluid poroviscoelastic media, as demonstrated by Kuan and Canic, who also tackled the nonlinear coupling at the interface \cite{osti_10502619}. Cesmelioglu analysed the weak formulation for NSBiot problems under the infinitesimal deformation circumstance, proofing existence and uniqueness of the solution, and providing a priori estimate \cite{cesmelioglu2017analysis}. Yotov and colleagues have made substantial contributions by developing an augmented dual mixed finite element formulation for the NSBiot problem, which includes the weak imposition of coupling conditions and establishes the existence and uniqueness of solutions for this formulation \cite{Yotov2023}. Additionally, they introduced a cell-centered finite volume approach based on a fully mixed formulation, which incorporates weakly symmetric stresses \cite{Yotov2020_fv}. From a numerical standpoint, two primary approaches have been employed to solve the NSBiot problem: monolithic solvers  \cite{fluids7070222, ambartsumyan2018lagrange, wen2020strongly,Buka202404} and partitional solvers based on domain decomposition techniques  \cite{BUKAC2015138, fluids7070222, guo2022decoupled, cesmelioglu2016optimization, MartinaOyekole, seboldt2021numerical}. 

Monolithic solvers are advantageous in handling complex interactions and coupling conditions within the model, particularly in scenarios involving significant added mass effects. In \cite{fluids7070222}, Wang et al. designed a second-order accurate FPSI monolithic solver, which is shown to be unconditionally stable, to study blood plasma flow in a prototype model of an implantable Bioartificial Pancreas. Ambartsumyan and Khattatov proposed a monolithic numerical solver for coupled Stokes-Biot model by employing a Lagrange multiplier method to impose the interaction conditions \cite{ambartsumyan2018lagrange}. In \cite{wen2020strongly}, an interior penalty discontinuous Galerkin method for the rearranged Stokes-Biot system was developed based on mixed finite element method. Furthermore, stability and priori error estimates were derived by Wen and He in \cite{wen2020strongly}. However, they come with challenges such as non-modular implementation and high memory requirements due to the simultaneous solving of all unknown variables. 

On the other hand, domain decomposition techniques, which include partitioned methods, offer a more modular approach by solving the fluid and structural subproblems separately. Given that both quick algorithms for (Navier-)Stokes equations \cite{guermond2006overview, wittum1989multi, li2013nonoverlapping} and parameter-stable methods for Biot equations \cite{adler2020robust, zeng2019h, gu2023iterative} have been studied for years, solving two subproblems independently may enable ones to leverage legacy codes and foundational work of predecessors. These methods typically involve iterative coupling through transmission conditions but may face stability issues due to added-mass effects, necessitating special treatments like implicit partitioned procedures or loosely coupled schemes to ensure accurate and stable solutions \cite{Martina2014}. 
A decoupled stabilized finite element method for the unsteady NSBiot problem was proposed by using the lowest equal-oeder finite elements in the work of Guo and Chen \cite{guo2022decoupled}. Cesmelioglu and Lee developed a decoupling approach by formulating the coupled NSBiot system as a constrained optimization problem, utilizing a Neumann-type control to enforce the continuity of normal stress across the interface \cite{cesmelioglu2016optimization}.
Notably, Bukac et al. have advanced the numerical analysis of fluid-structure interactions within poroelastic and poroviscoelastic materials by introducing two non-iterative Robin-Robin partitioned schemes for Stokes-Biot problems in fixed domain that achieve second-order temporal accuracy \cite{MartinaOyekole}. In addition, Bukac et al. also proposed two partitioned and loosely coupled methods based on the generalized Robin-Robin boundary conditions for NSBiot problems in moving domains \cite{seboldt2021numerical}. The partitioned procedures based on Robin-Robin conditions have shown particular promise in addressing large added-mass effect problems, making them a valuable tool in the numerical simulation of complex fluid-structure interactions.

Our work has been motivated by a series of open questions regarding the splitting methods mentioned below:
\begin{questions}
    \item Decoupled methods and the Robin-Robin type schemes have been extensively studied for Stokes-Darcy system over the past few decades \cite{chen2011parallel, sun2021domain, shi2023ensemble}. Given that coupling conditions for NSBiot problems are reminiscent of interface conditions in Stokes-Darcy system, could those established approaches offer valuable insights for our work? \label{Q1}
    \item Domain decomposition schemes are highly effective for accelerating computation and designing parallel algorithms, as they divide the original problem into smaller, manageable subproblems. After answering \ref{Q1}, could we leverage those techniques to design a parallelizable, loosely coupled scheme for FPSI problem? \label{Q2}
    \item The energy estimates for some existing methods rely on the viscoelastic model assumption \cite{seboldt2021numerical}. Following immediately from \ref{Q2}, is it possible to derive a strong stability for this parallel scheme without relying on viscoelasticity assumption? \label{Q3}
    \item We utilize an implicit scheme based on backward Euler to discretize the time derivatives in our NSBiot system. In terms of accuracy for the proposed scheme in \ref{Q2}, can we achieve the desired temporal first-order accuracy without requiring sub-iterations at each time step? \label{Q4}
\end{questions}
In the following sections, we will provide clear answers to all of the questions mentioned above. To model fluid flow in a moving domain, we utilize the Navier-Stokes (NS) equations in the Arbitrary Lagrangian-Eulerian (ALE) formulation. Building on the work of He et al. \cite{cao2014parallel}, we propose a parallel, non-iterative decoupling scheme for solving the NS-Biot problem, offering several advantages over traditional partitioned methods. First, our approach employs a loosely coupled scheme, significantly reducing the computational cost typically associated with sub-iterations. Second, it allows for the fluid and structure subproblems to be solved in parallel, substantially enhancing computational speed and efficiency, particularly suitable for large-scale simulations. Moreover, our scheme exhibits superior stability property, remaining robust even in the presence of challenges such as large added-mass effects. Unlike existing methods that require restrictive time step, our stability analysis shows that our scheme supports larger time steps, with the time step being independent of problem parameters like the storage coefficient $C_0$. Lastly, numerical experiments on benchmark problems demonstrate that the method achieves first-order accuracy in time for Stokes-Biot problems.

The rest of the paper is organized as follows. In Section \ref{sec:main2}, we introduce the mathematical model for NSBiot problems. Section \ref{sec:main3} presents the derivation of our parallel and loosely coupled numerical scheme, answering \ref{Q1} and \ref{Q2}. In Section \ref{sec:main4}, we demonstrate the unconditional stability of the proposed scheme in a linearized model scenario, giving an answer to \ref{Q3}. For the last question \ref{Q4}, two benchmark numerical examples are presented in Section \ref{sec:main5} to show the accuracy of the proposed scheme. Additionally, $2D$ and $3D$ examples are provided to further illustrate the method's stability and robustness. Finally, Section \ref{sec:main6} concludes this paper with several future research prospects.
\section{Problem description}
\label{sec:main2}
\subsection{Computational domains and mappings} 
We consider the interaction between an incompressible, Newtonian fluid and a poroelastic material in a moving domain. As shown in Figure \ref{fig:sketch}, let $\hat\Omega$ represents the union of the reference domain for the fluid subproblem $\hat\Omega_f$, the reference domain for the poroelastic material $\hat\Omega_p$, and the interface $\hat\Gamma$ that separates the fluid and poroelastic medium:
$\hat\Omega=\hat\Omega_f\cup\hat\Omega_p\cup\hat\Gamma$. The fluid and Biot reference external boundaries are denoted by $\hat\Sigma^D_f\cup\hat\Sigma^N_f$ and $\hat\Sigma^D_p\cup\hat\Sigma^N_p$, respectively. Since the porous medium is elastic and deformable, these domains will evolve over time, leading to the time-dependent domain $\Omega(t)=\Omega_f(t)\cup\Omega_p(t)\cup\Gamma(t)$. We assume that both regions are regular, bounded domains in $\mathbb R^d$ for either $d=2$ or 3. 
\begin{figure}[ht]
    \centering
    \includegraphics[width=0.5\linewidth]{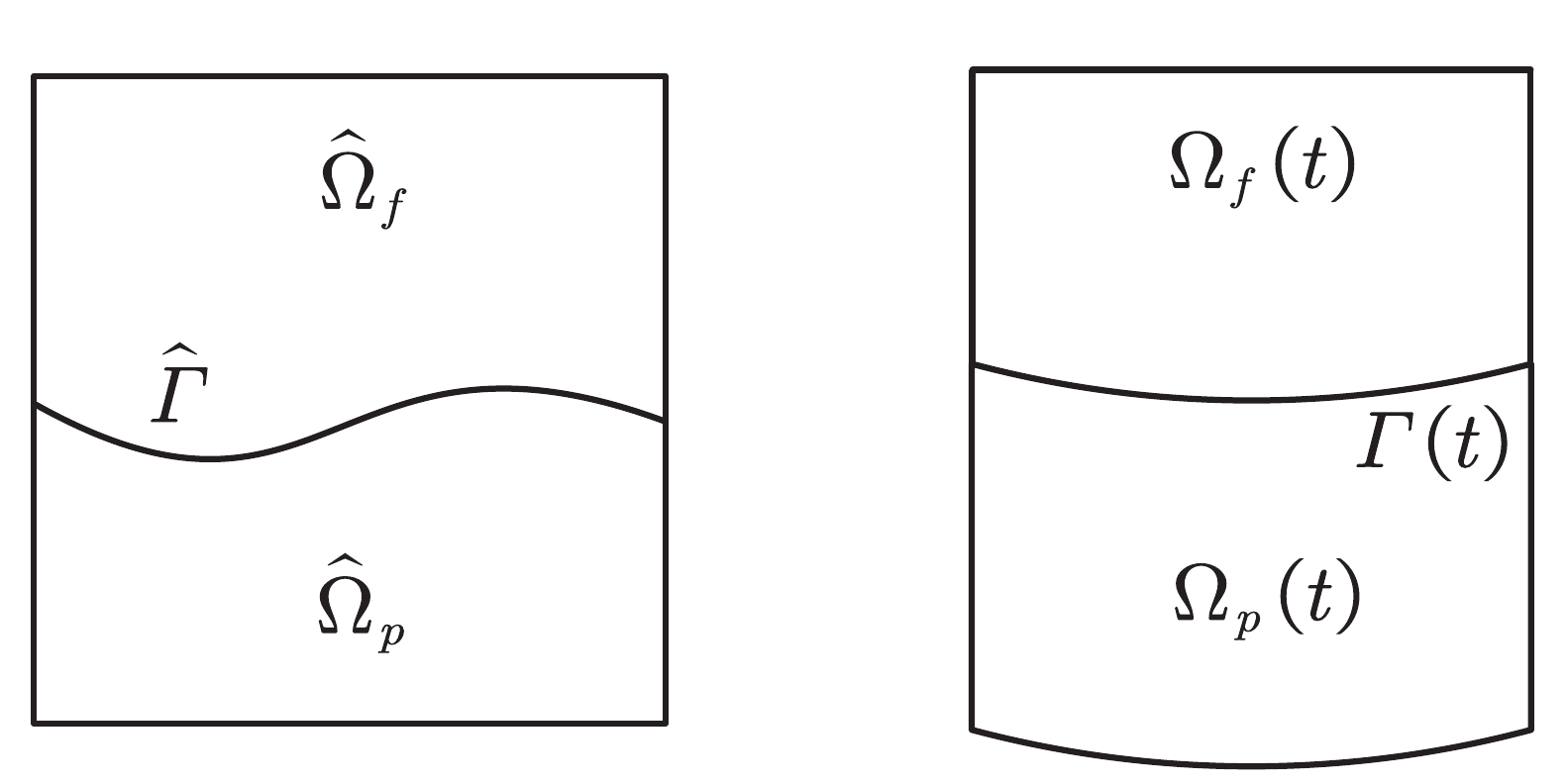}
    \caption{A sketch of the fluid-poroelastic structure interaction domain.}
    \label{fig:sketch}
\end{figure}

Let $\hat{\boldsymbol\eta}:[0, T]\times\hat\Omega_p\to\mathbb{R}^d$ be the displacement of the porous medium. We introduce the Lagrangian map to establish a connection between the reference domain $\hat\Omega_p$ and current configuration $\Omega_p(t)$ as follows:
$$
\boldsymbol\Phi_p(\hat{\boldsymbol{x}}, t)=\hat{\boldsymbol{x}}+\hat{\boldsymbol{\eta}}(\hat{\boldsymbol{x}}, t) \quad \text { for all } \hat{\boldsymbol{x}} \in \hat{\Omega}_P, ~t \in[0, T].
$$
The Lagrangian map is assumed to be smooth and invertible. In this case, for any function $\hat h:\hat{\Omega}_p\times[0, T]\to\mathbb R^d$, we introduce its counterpart in the Eulerian frame as $h=\hat h\circ\boldsymbol{\Phi}_p^{-1}$, defined as follows:
$$
h:\Omega_p(t)\to\mathbb R^d, \quad h(x, t) = \hat{h}\left(\boldsymbol{\Phi}_p^{-1}(x, t), t\right).
$$
To monitor the deformation of the fluid domain over time, we introduce a family of smooth and invertible Arbitrary Lagrangian Eulerian (ALE) mappings $\boldsymbol{\Phi}_f: \hat{\Omega}_f \rightarrow \Omega_f(t)$ that map the fixed reference domain onto the current, physical configuration as:
$$
\boldsymbol{\Phi}_f^t(\hat{\boldsymbol{x}})=\hat{\boldsymbol{x}}+\hat{\boldsymbol{\eta}}_f(\hat{\boldsymbol{x}}, t) \quad \text { for all } \hat{\boldsymbol{x}} \in \hat{\Omega}_f,~ t \in[0, T] \text {, }
$$
where $\hat{\boldsymbol{\eta}}_f$ represents the fluid domain displacement such that $\hat{\boldsymbol \eta}=\hat{\boldsymbol \eta}_f$ on $\hat{\Gamma}$.
In terms of the construction for the ALE mappings, the displacement $\hat{\boldsymbol{\eta}}_f$ can be extended arbitrarily from interface $\hat{\Gamma}$ into fluid domain $\hat\Omega_f$. In this work, we take the prevailing harmonic extension, i.e.,
\begin{equation}
    \hat\Delta\hat{\boldsymbol{\eta}}_f = 0 \quad\text{in }\hat\Omega_f, \quad \hat{\boldsymbol{\eta}}_f = \hat{\boldsymbol{\eta}}(t) \quad\text{on } \hat{\Gamma},\quad \hat{\boldsymbol{\eta}}_f = \boldsymbol{0} \quad\text{on }\partial\hat\Omega_f\setminus\hat\Gamma.
\end{equation}
to obtian a series of ALE mappings $\boldsymbol{\Phi}_f^t$. For any function $f: \Omega_f(t) \times[0, T] \rightarrow \mathbb{R}^d$, we denote its counterpart in the ALE frame as $\hat{f}=f \circ\boldsymbol\Phi_f^t$, defined as follows:
$$
\hat{f}: \hat{\Omega}_f \times[0, T] \rightarrow \mathbb{R}^d, \quad \hat{f}(\hat{\boldsymbol{x}}, t)=f\left(\boldsymbol{\Phi}_f^t(\hat{\boldsymbol{x}}, t), t\right) .
$$

\begin{remark}
In the remainder of this paper, we denote all entities defined in the reference domain with a hat $\hat{\cdot}$ and use the same notation without the hat for the Eulerian variables.
\end{remark}

\subsection{Fluid equation}
We consider the fluid to be an incompressible, viscous Newtonian fluid. The Navier-Stokes equations in the ALE framework are given by:
\begin{equation}
\left\{
\begin{aligned}
    & \rho_f\left(\left.\partial_t \boldsymbol{u}\right|_{\hat{\Omega}_f}+(\boldsymbol{u}-\boldsymbol{w}) \cdot \nabla \boldsymbol{u}\right)
    =\nabla \cdot \sigma_f\left(\boldsymbol{u}, p\right)+\boldsymbol{F}_f \quad&&\text{in}\ \Omega_f(t)\times(0, T),\\
    &\nabla \cdot \boldsymbol{u}=0\quad&&\text{in}\ \Omega_f(t)\times(0, T),\\
    & \boldsymbol{\sigma}_f \boldsymbol{n}_f=\boldsymbol{g} \quad&&\text{on}\ \Sigma_f^N(t) \times(0, T),\\
    & \boldsymbol{u}=\boldsymbol{0} \quad&&\text{on}\ \Sigma_f^D(t) \times(0, T),
\end{aligned}
\right.
\label{Navier-Stokes}
\end{equation}
where $\boldsymbol{u}$ denotes the fluid velocity, $\boldsymbol{w}$ is the domain velocity. $\rho_f$ is the fluid density, $\boldsymbol{n}_f$ is the outward normal vector, and $\boldsymbol{F}_f$ and $\boldsymbol{g}$ are the forcing terms. For a Newtonian fluid, the stress tensor $\boldsymbol{\sigma}_f$ is given by $\boldsymbol{\sigma}_f\left(\boldsymbol{u}, p\right)=-p \boldsymbol{I}+2 \mu_f \mathbb{D}(\boldsymbol{u})$, where $p_f$ is the fluid pressure, $\mu_f$ is the fluid viscosity, and $\mathbb{D}(\boldsymbol{u})=$ $\left(\nabla \boldsymbol{u}+(\nabla \boldsymbol{u})^T\right) / 2$ is the strain rate tensor.
\subsection{Structure equation}
We assume that the structure material is isotropic and homogeneous. Let $\hat\phi$ denote the fluid pore pressure, $\hat{\boldsymbol{u}}_p$ the filtration velocity of the fluid relative to the velocity of the solid and $\mathbb K$ the hydraulic conductivity tensor; $\hat{\boldsymbol{\eta}}$ denotes the solid displacement, $\hat{\boldsymbol{\xi}}=\partial_t\hat{\boldsymbol{\eta}}$ the solid velocity, $C_0$ the storativity coefficient, and $\alpha$ the Biot-Willis parameter. Then, the poroelastic structure assuming infinitesimal deformation can be described by the Biot model in the reference domain $\hat \Omega_p$ as:
\begin{equation}
\label{Biot}
\left\{
    \begin{aligned}
        &\rho_p\frac{\partial^2\hat{\boldsymbol{\eta}}}{\partial t^2} = \hat\nabla\cdot\hat{\boldsymbol{\sigma}}_p(\hat{\boldsymbol{\eta}},\hat\phi) \quad&&\text{in}\ \hat\Omega_P\times(0, T), \\
        & \hat{\boldsymbol{u}}_p = -\mathbb K\hat\nabla\hat\phi \quad&& \text{in}\ \hat\Omega_P\times(0, T), \\
        &C_0\frac{\partial\hat\phi}{\partial t} + \alpha\hat\nabla\cdot\hat{\boldsymbol{\xi}} - \hat\nabla\cdot(\mathbb K\hat\nabla\hat\phi) = \hat{F}_d \quad&& \text{in}\ \hat\Omega_P\times(0, T), \\
        &\hat{\boldsymbol{\sigma}}_p\hat{\boldsymbol{n}}_p = 0 \quad&&\text{on}\ \hat\Sigma_{p}^{N}\times(0, T), \\
        &\hat{\boldsymbol{\eta}} = \partial_{t}\hat{\boldsymbol{\eta}}= 0 \quad&&\text{on}\ \hat\Sigma_{p}^{D}\times(0, T),
    \end{aligned}
    \right.
\end{equation}
where $\hat{\boldsymbol{\sigma}}_p = \hat{\boldsymbol{\sigma}}_s(\hat{\boldsymbol{\eta}})-\alpha\hat\phi\boldsymbol{I}$ denotes the Cauchy stress tensor for the poroelastic structure; The elastic stress tensor $\hat{\boldsymbol{\sigma}}_s$ is characterized by the Saint Venant-Kirchhoff constitutive model
$$
\hat{\boldsymbol{\sigma}}_s(\hat{\boldsymbol{\eta}})=2\mu_p\hat{\mathbb D}(\hat{\boldsymbol{\eta}})+\lambda_p\hat\nabla\cdot\hat{\boldsymbol{\eta}}\boldsymbol{I},
$$
where $\mu_p$ and $\lambda_p$ are Lamé parameters. $\hat F_d$ is the volumetric source term. The hydraulic conductivity tensor $\mathbb K$ is assumed to be uniformly bounded and uniformly elliptic, i.e., for $\forall x\in \hat{\Omega}_p$ and $\forall \boldsymbol z\in\mathbb R^d$, there exist positive constants $K_m$ and $K_M$ such that
\begin{equation}
    K_m\Vert \boldsymbol z\Vert^2\le \boldsymbol z^T\mathbb K(x) \boldsymbol z\le K_M\Vert \boldsymbol z\Vert^2.
    \label{uniformly}
\end{equation}

To describe the Cauchy stress tensor $\hat{\boldsymbol{\sigma}}_p$ in Eulerian system, one applies the Piola transformation to reformulate its counterpart as follows:
$$
\boldsymbol{\sigma}_p = \hat{J}^{-1}\hat{\boldsymbol{\sigma}}_p\hat{\boldsymbol{F}}^T,
$$
where $\hat{\boldsymbol{F}}$ and $\hat{J}$ represent the deformation gradient and corresponding Jacobian of the Lagrangian map $\boldsymbol{\Phi}_p$ respectively. As mentioned before that the deformation of structure is assumed to be infinitesimal, the deformation gradient $\hat{\boldsymbol{F}}$ is therefore so small that we approximate $\hat{\boldsymbol{F}}=I$ and $\hat{J} = 1$ \cite[Section 2.3]{richter2017fluid}. By this simplification, we have the relationship as follows:
\begin{equation}
    \boldsymbol{\sigma}_p = \hat{\boldsymbol{\sigma}}_p\circ\boldsymbol{\Phi}_p^{-1}.
\end{equation}

\subsection{Coupling condition}
To couple the fluid model and the Biot model, we impose a set of coupling conditions on interface $\Gamma(t)$ as follows:
\begin{subequations}
\begin{align}
    &(\boldsymbol{\xi} + \boldsymbol{u}_p)\cdot\boldsymbol{n}_f = \boldsymbol{u}\cdot\boldsymbol{n}_f \quad&&\text{on}\ \Gamma(t)\times(0, T), \label{interface1}\\
    &\boldsymbol{\tau}_{f, j}\cdot\boldsymbol{\sigma}_f\boldsymbol{n}_f = -\gamma(\boldsymbol{u} - \boldsymbol{\xi})\cdot\boldsymbol{\tau}_{f, j}\quad \text{for } j = 1,\cdots, d-1 \quad&&\text{on}\ \Gamma(t)\times(0, T),\label{interface2}\\
    &\boldsymbol{n}_f\cdot\boldsymbol{\sigma}_f\boldsymbol{n}_f = -\phi\quad&&\text{on}\ \Gamma(t)\times(0, T), \label{interface3}\\
    &\boldsymbol{\sigma}_f\boldsymbol{n}_f =  \boldsymbol{\sigma}_p\boldsymbol{n}_f\quad&&\text{on}\ \Gamma(t)\times(0, T)\label{interface4},
\end{align}
\label{interface condition}
\end{subequations}
where $\boldsymbol{\tau}_{f, j}$ denotes an orthonormal sets of unit vectors on the tangential plane to $\Gamma(t)$; $\gamma > 0$ denotes the slip rate in the Beavers-Joseph-Saffman (BJS) interface condition.  

\subsection{Weak Formation}
For the Navier-Stokes-Biot system given by (\ref{Navier-Stokes}), (\ref{Biot}) and (\ref{interface condition}), we consider the following functional spaces for all $t\in[0, T]$: 
\begin{align*}
V_{f}(t) &= \{\boldsymbol{v}:\Omega_{f}(t)\rightarrow \mathbb{R}^{d}\ \big{|}\   \boldsymbol{v} = \hat{\boldsymbol{v}}\circ\left(\boldsymbol{\Phi}_{f}^t\right)^{-1}, \hat{\boldsymbol{v}}\in(H^{1}(\hat{\Omega}_f))^{d}, \hat{\boldsymbol{v}} = \boldsymbol{0}\ \text{on}\ \hat{\Sigma}_f^D\},\\%
Q_{f}(t) &= \{ q:\Omega_{f}(t)\rightarrow \mathbb{R}\ \big{|}\   q = \hat{q}\circ\left(\boldsymbol{\Phi}_{f}^t\right)^{-1}, \hat{p}\in(L^{2}(\hat{\Omega}_f)) \},\\%
\hat{V}_{p} &= \{ \hat{\boldsymbol{\zeta}}:\hat{\Omega}_{p}\rightarrow \mathbb{R}^{d}\ \big{|}\, \hat{\boldsymbol{\zeta}}\in(H^{1}(\hat{\Omega}_{p}))^{d}, \hat{\boldsymbol{\zeta}} = \boldsymbol{0}\ \text{on}\ \hat{\Sigma}_p^D \},\\%
\hat{Q}_{p} &=\{\hat{\psi}:\hat{\Omega}_{p}\rightarrow \mathbb{R}\ \big{|} \  \hat{\psi}\in(H^{1}(\hat{\Omega}_{p})), \hat{\psi} = 0\ \text{on}\ \hat{\Sigma}_p^D \},\\
\hat{X}_{p} &= \{ \hat{\boldsymbol{v}}_p:\hat{\Omega}_{p}\rightarrow \mathbb{R}^{d}\ \big{|}\, \hat{\boldsymbol{v}}_p\in(L^{2}(\hat{\Omega}_{p}))^{d} \},
\end{align*}
where $H^{1}(\Omega_f)$ and $H^1(\Omega_p)$ denote the usual Sobolev spaces. For simplicity of notation, let $(\cdot,\cdot)_f$ and $(\cdot,\cdot)_p$ denote integration on the current fluid domain and reference Biot domain respectively; let $\langle\cdot,\cdot\rangle_\Gamma$ and $\langle\cdot,\cdot\rangle_{\hat\Gamma}$denotes the boundary integration on the interface $\Gamma(t)$ and $\hat{\Gamma}$ respectively. In addition, let $P_f$ denote the projection from the fluid domain onto the tangential space on $\Gamma(t)$, such that:
$$
    P_f(\boldsymbol{u}) = \sum_{j=1}^{d-1}(\boldsymbol{u}\cdot\boldsymbol{\tau}_{f, j})\boldsymbol{\tau}_{f,j}. 
$$
Similarly, $\hat P_p$ is defined to denote the projection from the reference poroelastic domain onto the tangential space on $\hat\Gamma$.

With these notations, the weak formulation of the coupled Navier-Stokes-Biot problem is given as follows: find $\boldsymbol{u}\in V_f(t)$, $p\in Q_f(t)$, $\boldsymbol{\hat{\xi}}=d_t\boldsymbol{\hat{\eta}}\in\hat{V}_{p}$ and $\hat{\phi}\in \hat{Q}_{p}$ such that
\begin{equation}
    \begin{aligned}
        &\rho_f\left[\left(\left.\frac{\partial\boldsymbol{u}}{\partial t}\right|_{\hat{\boldsymbol{x}}}, \boldsymbol{v}\right)_f 
        + \left( (\boldsymbol{u-\omega})\cdot\nabla\boldsymbol{u},\boldsymbol{v}\right)_f\right]
        +2\mu_{f}\left(\mathbb D(\boldsymbol{u}) ,\mathbb D(\boldsymbol{v})\right)_f
        - \left(p, \nabla\cdot\boldsymbol{v}\right)_f\\
        &+ \left(\nabla\cdot\boldsymbol{u}, q\right)_f+ \left\langle\phi,\boldsymbol{v}\cdot\boldsymbol{n}_{f}\right\rangle_\Gamma
        + \gamma\left\langle P_{f}\left(\boldsymbol{u}-\boldsymbol{\xi}\right), P_{f}\left(\boldsymbol{v}\right)\right\rangle_\Gamma\\
        &+ \rho_{p}\left(\frac{\partial\boldsymbol{\hat\xi}}{\partial t}, 
        \boldsymbol{\hat\zeta}\right)_p
        + 2\mu_{p}\left(\hat{\mathbb D}(\boldsymbol{\hat\eta}), \hat{\mathbb D}(\boldsymbol{\hat\zeta})\right)_p 
        + \lambda_{p}\left(\hat\nabla\cdot\boldsymbol{\hat\eta}, \hat\nabla\cdot\boldsymbol{\hat\zeta}\right)_p
        -\alpha\left(\hat\phi, \hat\nabla\cdot\boldsymbol{\hat\zeta}\right)_p \\
        &+ C_0\left(\frac{\partial\hat\phi}{\partial t},\hat\psi\right)_p
        + \alpha\left(\hat\nabla\cdot\hat{\boldsymbol{\xi}}, \hat\psi\right)_p
        + \left(\mathbb{K}\hat\nabla\hat\phi, \hat\nabla\hat\psi\right)_p \\
        &- \gamma\left\langle \hat P_{p}\left(\hat{\boldsymbol{u}}-\hat{\boldsymbol{\xi}} \right) , \hat P_{p}\left(\hat{\boldsymbol{\zeta}}\right)\right\rangle_{\hat\Gamma}
        + \left\langle\left(\hat{\boldsymbol{\xi}} - \hat{\boldsymbol{u}}\right)\cdot\hat{\boldsymbol{n}}_{p}, \hat\psi\right\rangle_{\hat\Gamma}+ \left\langle\hat\phi ,\hat{\boldsymbol{\zeta}}\cdot\hat{\boldsymbol{n}}_{p}\right\rangle_{\hat\Gamma} \\
        &= \left(\boldsymbol{F}_{f},\boldsymbol{v} \right)_f
        + \left(F_{d},\psi \right)_p, 
        \qquad \forall \left(\boldsymbol{v}, q, \hat{\boldsymbol{\zeta}}, \hat\psi\right) \in\boldsymbol{V}_f(t)\times Q_{f}(t)\times\hat{\boldsymbol{V}}_{p}\times \hat{Q}_{p}.
\end{aligned}
\label{coupled}
\end{equation}
\section{The loosely coupled splitting scheme}\label{sec:main3}
\subsection{Robin-type boundary conditions}
In order to solve the coupled Naiver-Stokes-Biot problem utilizing domain decomposition, we introduce Robin bou

ndary conditions for the NS and Biot system. We denote solutions of our decoupled method with a subscript $\cdot_r$. 
To formulate  Robin-type boundary conditions for normal component and tangential component, we introduce the following two boundary conditions for fluid sub-problem:
\begin{align}
&\boldsymbol{n}_f\cdot\boldsymbol{\sigma}_{fr}\boldsymbol{n}_f + L_1\boldsymbol{u}_r\cdot\boldsymbol{n}_f = R_1 \quad&&\text{on }\Gamma(t)\label{R1}\\
&P_f\left(\boldsymbol{\sigma}_{fr}\boldsymbol{n}_f\right)+\gamma P_f\left(\boldsymbol{u}_r\right)=
    R_{2}
    \quad&& \text{on } \Gamma(t).\label{R2}
\end{align}
where $R_1$ and $R_2$ denotes two known functions defined on $\Gamma(t)$. In this case, the corresponding weak formation for the uncoupled Navier-Stokes system (\ref{Navier-Stokes}) is given by the following: for $R_{1}, R_{2}\in L^{2}(0,T;L^{2}(\Gamma(t)))$, find $\boldsymbol{u}_r\in V_{f}(t)$, $p_r\in Q_{f}(t)$ such that:
\begin{equation}
\begin{aligned}
    &\rho_f\left(\left.\frac{\partial\boldsymbol{u}_r}{\partial t}\right|_{\hat{\boldsymbol{x}}}, \boldsymbol{v}\right)_f 
        + \rho_f\left( (\boldsymbol{u}_r-\boldsymbol{\omega})\cdot\nabla\boldsymbol{u}_r,\boldsymbol{v}\right)_f
        +2\mu_{f}\left(\mathbb D(\boldsymbol{u}_r) ,\mathbb D(\boldsymbol{v})\right)_f
        - \left(p_r, \nabla\cdot\boldsymbol{v}\right)_f\\
        &+ L_{1}\left\langle\boldsymbol{u}_r\cdot\boldsymbol{n}_f, \boldsymbol{v}\cdot\boldsymbol{n}_f\right\rangle_\Gamma + \gamma\left\langle P_f(\boldsymbol{u}_r), P_f(\boldsymbol{v})\right\rangle_\Gamma
        = \left(\boldsymbol{F}_{f}, \boldsymbol{v}\right)_{f}
        + \left\langle  R_{2}, P_f(\boldsymbol{v})\right\rangle_\Gamma\\
        &+\left\langle R_{1},\boldsymbol{v}\cdot\boldsymbol{n}_f\right\rangle_\Gamma
        +\left\langle \boldsymbol{g},\boldsymbol{v}\right\rangle_{\Sigma_f^N}
        \qquad\forall\boldsymbol{v}\in\boldsymbol{V}_f(t),\\
        &\left(\nabla\cdot\boldsymbol{u}_r, q\right)_f = 0 \qquad\forall q\in Q_f(t).
\end{aligned}
\label{ns_robin1}
\end{equation}
Now, we introduce a general Robin-type boundary condition in normal direction for the first equation, which characterizes the elastic deformation, in Biot system as:
\begin{equation}
\hat{\boldsymbol{n}}_p\cdot\hat{\boldsymbol{\sigma}}_{pr}\hat{\boldsymbol{n}}_p +\hat{\phi}_r + L_2\hat{\boldsymbol{\xi}}_r\cdot\hat{\boldsymbol{n}}_p
=R_{3}\\ 
\quad\text{on } \hat{\Gamma},
\label{R3}
\end{equation}
where $R_3$ above denotes a known function defined on $\hat{\Gamma}$. Another Robin-type condition for the equation characterizing porous flows in Biot system is proposed as:
\begin{equation}
\mathbb{K}\hat\nabla\hat{\phi}_r\cdot\hat{\boldsymbol{n}}_p + L_{3}\hat{\phi}_r - \hat{\boldsymbol{\xi}}_{r}\cdot\hat{\boldsymbol{n}}_p
=R_{4} \quad\text{on } \hat{\Gamma},
\label{R4}
\end{equation}
For the Biot system, by combining coupling condition (\ref{interface2}) and (\ref{interface4}) one yields a Robin-type condition in tangential directions as well:
\begin{equation}
    \hat{P}_p\left(\hat{\boldsymbol{\sigma}}_{pr}\hat{\boldsymbol{n}}_p\right) + \gamma\hat{P}_p\left(\hat{\boldsymbol{\xi}}_r\right) = R_{5} \quad\text{on }\hat\Gamma.
    \label{R5}
\end{equation}
where $R_{4}$ and $R_{5}$ , similar to $R_3$, represents two known functions defined on $\hat{\Gamma}$ as well. After embedding these proposed conditions to Biot equation, we can rewrite the weak formulation for the uncoupled Biot system as: given $R_{3}, R_{4}, R_{5}\in L^2(0, T;L^2(\Gamma))$, find $\hat{\boldsymbol{\xi}}_r=\partial_t\hat{\boldsymbol{\eta}}_r\in\hat{{V}}_{p}$ and $\hat{\phi}_r\in\hat{Q}_{p}$ such that
\begin{align}
        &\rho_{p}\left(\frac{\partial\hat{\boldsymbol{\xi}}_r}{\partial t}, \hat{\boldsymbol{\zeta}}\right)_p
        + 2\mu_{p}\left(\hat{\mathbb{D}}(\hat{\boldsymbol{\eta}}_r), \hat{\mathbb{D}}(\hat{\boldsymbol{\zeta}})\right)_p
        + \lambda_{p}\left(\hat{\nabla}\cdot\hat{\boldsymbol{\eta}}_r, \hat{\nabla}\cdot\hat{\boldsymbol{\zeta}}\right)_p
        -\alpha\left(\hat{\phi}_r, \hat\nabla\cdot\hat{\boldsymbol{\zeta}}\right)_p\nonumber \\
        & + \gamma\left\langle \hat{P}_p\left(\hat{\boldsymbol{\xi}}_r\right),\hat{P}_p\left(\hat{\boldsymbol{\zeta}}\right)\right\rangle_{\hat{\Gamma}} + \left\langle\hat{\phi}_r+L_2 \hat{\boldsymbol{\xi}}_r\cdot\hat{\boldsymbol{n}}_p, \hat{\boldsymbol{\zeta}}\cdot\hat{\boldsymbol{n}}_p\right\rangle_{\hat{\Gamma}}
        = \left\langle R_{3}, \hat{\boldsymbol{\zeta}}\cdot\hat{\boldsymbol{n}}_p\right\rangle_{\hat{\Gamma}}\nonumber \\
&+\left\langle  R_{5}
,\hat{P}_p\left(\hat{\boldsymbol{\zeta}}\right)\right\rangle_{\hat{\Gamma}} \qquad\forall\boldsymbol{\zeta}\in\hat{\boldsymbol{V}}_{p},\label{biot_robin1}\\
        &C_0\left(\frac{\partial\hat{\phi}_r}{\partial t},\hat{\psi}\right)_p
        + \alpha\left(\hat\nabla\cdot\hat{\boldsymbol{\xi}}_r, \hat\psi\right)_p
        + \left(\mathbb{K}\hat\nabla\hat{\phi}_r, \hat\nabla\hat\psi\right)_p
        + L_{3}\left\langle\hat{\phi}_r, \hat\psi\right\rangle_{\hat{\Gamma}}
        = \left(\hat{F}_d, \hat\psi\right)_p\nonumber \\
        &+ \left\langle R_{4}, \hat\psi\right\rangle_{\hat{\Gamma}} \forall\hat\psi\in \hat{Q}_p.\nonumber
\end{align}

It is worth noting that for specified functions $R_1\sim R_5$, weak formulations (\ref{ns_robin1}) and (\ref{biot_robin1}) are \textit{uncoupled}. This characteristic facilitates the development of a decoupled algorithm, enabling the separate computation of these two subproblems. In the following theorem, we will show that for approximate choices of $L_1\sim L_3$ and $R_1\sim R_5$, solutions of the original NSBiot system (\ref{coupled}) are equivalent to solutions of (\ref{ns_robin1}) and (\ref{biot_robin1}), and therefore we may solve the latter problems instead of original ones.

\begin{theorem}[Equivalency]
        \label{Equivalency}
        Assume that the solution of the NSBiot system in moving domain is unique. Let $\left(\boldsymbol{u}, p,\hat{\boldsymbol{\eta}},\hat{\boldsymbol{\xi}},\hat\phi\right)$ be
        the solution of the coupled NSBiot system (\ref{coupled}), and let
        $\left(\boldsymbol{u}_r, p_r,\hat{\boldsymbol{\eta}}_r,\hat{\boldsymbol{\xi}}_r,\hat{\phi}_r \right)$
        be the solution of the decoupled NSBiot system (\ref{ns_robin1}) and (\ref{biot_robin1}) with general Robin-type boundary conditions at the interface. Then $\left(\boldsymbol{u}, p,\boldsymbol{\eta},\boldsymbol{\xi},\phi\right) = 
        \left(\boldsymbol{u}_r, p_r,\hat{\boldsymbol{\eta}}_r,\hat{\boldsymbol{\xi}}_r,\hat{\phi}_r \right)$
        if and only if $R_1, R_2, R_3, R_{4}$ and $R_{5}$ satisfy the  following compatibility conditions 
\begin{subequations}
\label{compatibility}
\begin{align}
    R_{1} & = L_{1}\boldsymbol{u}_r\cdot\boldsymbol{n}_{f} - \phi_r\quad&&\text{on}\ \Gamma(t),\label{c1}\\
    R_{2}  &= \gamma\boldsymbol{P}_{f}(\boldsymbol{\xi}_{r}) \quad&&\text{on}\ \Gamma(t),\label{c2}\\
    R_{3} & = L_2\hat{\boldsymbol{\xi}}_r\cdot\hat{\boldsymbol{n}}_p \quad&&\text{on}\ \hat\Gamma,\\
    R_{4} & = L_{3}\hat\phi_r  - \hat{\boldsymbol{u}}_r \cdot\hat{\boldsymbol{n}}_p 
    \quad&&\text{on}\ \hat\Gamma. \\
    R_{5}  &= \gamma\hat{\boldsymbol{P}}_{p}(\hat{\boldsymbol{u}}_{r})\quad&&\text{on}\ \hat\Gamma.
\end{align}
\end{subequations}
\end{theorem}
\begin{proof}
    For the necessity, we set $\left(\hat{\boldsymbol{\zeta}}, \hat{\psi}\right) = 0$ in the NSBiot system (\ref{coupled}) and deduce that $\left(\boldsymbol{u}, p, \boldsymbol{\hat{\eta}}, \boldsymbol{\hat{\xi}}, \hat{\phi}\right)$ solves (\ref{ns_robin1}) if:
    $$
    \langle R_{1} - L_{1}\boldsymbol{u}_r\cdot\boldsymbol{n}_f+\phi_r, \boldsymbol{v}\cdot\boldsymbol{n}_{f}\rangle_\Gamma + \langle R_2 - \gamma\boldsymbol{P}_f(\boldsymbol{\xi}_r), \boldsymbol{P}_f(\boldsymbol{v})\rangle_\Gamma = 0\quad \forall \boldsymbol{v}\in\boldsymbol{V}_{f}(t)
    $$
    which implies (\ref{c1}) and (\ref{c2}). The necessity for $R_3, R_4, R_5$ can be derived in a similar fashion.
    In terms of sufficiency ,by substituting the compatibility conditions (\ref{compatibility}) in (\ref{ns_robin1}) and (\ref{biot_robin1}), one yields that $\left(\boldsymbol{u}_r, p_r,\hat{\boldsymbol{\eta}}_r,\hat{\boldsymbol{\xi}}_r,\hat{\phi}_r \right)$ solves the coupled NSBiot system (\ref{coupled}). As the solution of the NSBiot system is assumed to be unique, we have $\left(\boldsymbol{u}, p,\hat{\boldsymbol{\eta}},\hat{\boldsymbol{\xi}},\hat\phi\right) = 
        \left(\boldsymbol{u}_r, p_r,\hat{\boldsymbol{\eta}}_r,\hat{\boldsymbol{\xi}}_r,\hat{\phi}_r \right)$.
\end{proof}
\begin{remark}
    For NSBiot problems in fixed domain, this equivalency holds without any assumption as existence and uniqueness of solution are guaranteed in \cite{cesmelioglu2017analysis}.
\end{remark}
\begin{remark}
    Following the choice in \cite{chen2011parallel}, we set $L_1=\frac{1}{L3}=L$ because convergence is guaranteed for Stokes-Darcy problems in this case, and $L_2 = 1$ for simplicity. From our benchmark test in Section \ref{real}, however, no convergence and rate of error issues were observed in linear Stokes-Biot problems.
\end{remark}
\begin{algorithm}[htbp]  
  \caption{ Fully parallelizable splitting scheme for  NSBiot system with moving interface}  
  \label{semi} 
  Given $\boldsymbol{u}^n, \boldsymbol{\eta}^n, \boldsymbol{\xi}^n$ and $\phi^n$, $\boldsymbol{\omega}^n$, we solve for $\boldsymbol{u}^{n+1}, \boldsymbol{\eta}^{n+1}, \boldsymbol{\xi}^{n+1}$ and $\phi^{n+1}$,$\boldsymbol{\omega}^{n+1}$, for $n=0, 1, 2, \cdots, N-1.$
  
\begin{algorithmic}[htbp]
\STATE \textbf{Set}
    \begin{align*}
        R_{1}^n=&L\boldsymbol{u}^n\cdot\boldsymbol{n}_f^n-\phi^n,
        &&R_{2}^n = \gamma P_f^n\left(\boldsymbol{\xi}^n\right),\\
        R_{3}^n=&\hat{\boldsymbol{\xi}}^n\cdot\hat{\boldsymbol{n}}_p,
        &&R_{4}^n=-\hat{\boldsymbol{u}}^n\cdot\hat{\boldsymbol{n}}_p+\hat\phi^n/L,\\
        R_{5}^n =& \gamma\hat{P}_p\left(\hat{\boldsymbol{u}}^n\right).
    \end{align*}

\STATE \textbf{Solve the following two sub-problems in parallel.}

\vspace{0.2in}
\noindent  \underline{\textbf{Fluid Subproblem:}}

$<1>$ Solve ALE mappings $\boldsymbol{\Phi}_f^{t^n}=\hat{\boldsymbol{x}}+\hat{\boldsymbol{\eta}}_f^n$ by (\ref{ALE})
     \begin{equation}
    \hat\Delta\hat{\boldsymbol{\eta}}^n_f = 0 \quad\text{in }\hat\Omega_f, \quad \hat{\boldsymbol{\eta}}^n_f = \hat{\boldsymbol{\eta}}^n \quad\text{on } \hat{\Gamma},\quad \hat{\boldsymbol{\eta}}_f^n = 0 \quad\text{on }\partial\hat\Omega_f\setminus\hat\Gamma.
     \label{ALE}
     \end{equation}
     Moreover, calculate $\boldsymbol{\omega}^n$ such that $\boldsymbol{\omega}^n=d_t\hat{\boldsymbol{\eta}}_f^n$ and update fluid mesh by setting $\Omega_f^{n}=\boldsymbol{\Phi}_f^{t^n}(\hat{\Omega}_f)$.
     
$<2>$ Solve (\ref{fluid}) for $\boldsymbol{u}^{n+1}$ and $p^{n+1}$.
     \begin{equation}
     \begin{aligned}
         &\rho_f\left(\frac{\boldsymbol{u}^{n+1}-\boldsymbol{u}^n}{\Delta t}, \boldsymbol{v}\right)_{f^n}
        +\rho_f\left((\boldsymbol{u}^n - \boldsymbol{\omega}^n)\cdot\nabla\boldsymbol{u}^{n+1}, \boldsymbol{v}\right)_{f^n}
        + 2\mu_f\left(\mathbb D(\boldsymbol{u}^{n+1}), \mathbb{D}(\boldsymbol{v})\right)_{f^n}
        \\&-(p^{n+1}, \nabla\cdot\boldsymbol{v})_{f^n}
        +(\nabla\cdot\boldsymbol{u}^{n+1}, q)_{f^n}\\
        &+\gamma\langle P_f^n\boldsymbol{u}^{n+1}, P_f^n\boldsymbol{v}\rangle_{\Gamma^n}
        +L\langle\boldsymbol{u}^{n+1}\cdot\boldsymbol{n}_f
        ,\boldsymbol{v}\cdot\boldsymbol{n}_f\rangle_{\Gamma^n}
        =
       \langle R_{2}^{n}, P_f^n\boldsymbol{v}\rangle_{\Gamma^n}
        \\&+\left(\boldsymbol{F}^{n+1}_f, \boldsymbol{v}\right)_{f^n}
        +\langle R_1^n, \boldsymbol{v}\cdot\boldsymbol{n}_f\rangle_{\Gamma^n}
        +\langle \boldsymbol{g}^{n+1}, \boldsymbol{v}\rangle_{\Sigma_f^{N}}, 
        ~\forall (\boldsymbol{v}, q)\in (V_f, Q_f)
     \end{aligned}
     \label{fluid}
    \end{equation}
    
$<3>$ Move mesh back.

\vspace{0.2in}
\noindent \underline{\textbf{Biot Subproblem:}}

Solve (\ref{A2}) for $\partial_t\hat{\boldsymbol{\eta}}^{n+1}=\hat{\boldsymbol{\xi}}^{n+1}$ and $\hat\phi^{n+1}$.
\begin{equation}
     \begin{aligned}
      \label{A2}
        &\rho_p\left(\frac{\hat{\boldsymbol{\xi}}^{n+1}-\hat{\boldsymbol{\xi}}^n}{\Delta t}, \boldsymbol{\hat\zeta}\right)_{p}
        + 2\mu_p\left(\hat{\mathbb D}(\hat{\boldsymbol{\eta}}^{n+1}), \hat{\mathbb{D}}(\hat{\boldsymbol{\zeta}})\right)_{p}
        + \lambda_p(\hat{\nabla}\cdot\hat{\boldsymbol{\eta}}^{n+1}, \hat{\nabla}\cdot\hat{\boldsymbol{\zeta}})_{p}\\
        &-\alpha(\hat{\phi}^{n+1}, \hat{\nabla}\cdot\hat{\boldsymbol{\zeta}})_{p}
        + C_0\left(\frac{\hat\phi^{n+1}-\hat\phi^n}{\Delta t}, \hat\psi\right)_{p}
        +\alpha(\hat{\nabla}\cdot\hat{\boldsymbol{\xi}}^{n+1}, \hat\psi)_{p}
        +(\mathbb K\hat{\nabla}\hat{\phi}^{n+1},\hat{\nabla}\hat{\psi})_{p}\\
        &+\gamma\langle \hat{P}_p\hat{\boldsymbol{\xi}}^{n+1}, \hat{P}_p\hat{\boldsymbol{\zeta}}\rangle_{\hat{\Gamma}}
        +\frac{1}{L}\langle\hat{\phi}^{n+1},\hat{\psi}\rangle_{\hat{\Gamma}}
        +\langle\hat{\phi}^{n+1}, \hat{\boldsymbol{\zeta}}\cdot\hat{\boldsymbol{n}}_p\rangle_{\hat{\Gamma}}
        -\langle\hat{\boldsymbol{\xi}}^{n+1}\cdot\hat{\boldsymbol{n}}_p,\hat\psi\rangle_{\hat{\Gamma}}\\
        &=\langle R_{5}^{n}, \hat{P}_p\hat{\boldsymbol{\zeta}}\rangle_{\hat{\Gamma}} 
        +\left(\hat{F}^{n+1}_d,\hat{\psi}\right)_{p}
        +\langle R_3^n, \hat{\boldsymbol{\zeta}}\cdot\hat{\boldsymbol{n}}_p\rangle_{\hat{\Gamma}}
        +\langle R_4^n, \hat\psi\rangle_{\hat{\Gamma}}, \\
        &\forall (\hat{\boldsymbol{\zeta}}, \hat{\psi})\in (\hat{V}_p, \hat{Q}_p)
    \end{aligned}
    \end{equation}
\STATE \textbf{Advancing to the next timestep.}
\end{algorithmic}
\end{algorithm}
\subsection{Proposed parallel decoupling algorithm}
In the remainder of this manuscript, we will omit all $r$ subscripts that indicate solutions of the decoupled method. To solve the fluids and Biot subproblems concurrently in time, we set right-hand side for those Robin-type boundaries explicit while keep left-hand side implicit. Besides, to avoid additional nonlinearity, interface conditions for fluid are evaluated on $\Gamma(t^n)$, resulting in the following discrete approximation:
\begin{subequations}
\begin{align}
&\boldsymbol{n}_f^n\cdot\left(\boldsymbol{\sigma}_f\boldsymbol{n}_f\right)^{n+1} + L\boldsymbol{u}^{n+1}\cdot\boldsymbol{n}_f^n = L\boldsymbol{u}^n\cdot\boldsymbol{n}_f^n - \phi^n\quad&&\text{on}\ \Gamma(t^n), \\
&P_f^n\left[\left(\boldsymbol{\sigma}_f\boldsymbol{n}_f\right)^{n+1}\right]+\gamma P_f^n\left(\boldsymbol{u}^{n+1}\right)=\gamma P_f^n\left(\boldsymbol{\xi}^n\right)\quad&&\text{on}\ \Gamma(t^n), \\
&\hat{\boldsymbol{n}}_p\cdot\hat{\boldsymbol{\sigma}}_p^{n+1}\hat{\boldsymbol{n}}_p +\hat{\phi}^{n+1} + \hat{\boldsymbol{\xi}}^{n+1}\cdot\hat{\boldsymbol{n}}_p=\hat{\boldsymbol{\xi}}^n\cdot\hat{\boldsymbol{n}}_p\quad&&\text{on}\ \hat{\Gamma},\\
    &\mathbb{K}\hat\nabla\hat{\phi}^{n+1}\cdot\boldsymbol{n}_p + \hat{\phi}^{n+1}/L - \hat{\boldsymbol{\xi}}^{n+1}\cdot\hat{\boldsymbol{n}}_p =\hat\phi^n/L - \hat{\boldsymbol{u}}^n\cdot\hat{\boldsymbol{n}}_p\quad&&\text{on } \hat{\Gamma}.\\
&\hat{P}_p\left(\hat{\boldsymbol{\sigma}}_p^{n+1}\hat{\boldsymbol{n}}_p\right) + \gamma\hat{P}_p\left(\hat{\boldsymbol{\xi}}^{n+1}\right) = \gamma\hat{P}_p\left(\hat{\boldsymbol{u}}^n\right)\quad&&\text{on}\ \hat\Gamma.
\end{align}
\end{subequations}
By applying back Euler scheme for time derivatives, the weak form of the proposed parallel semi-discretized scheme is outlined in Algorithm \ref{semi}.
\section{Stability analysis of the linear system}
\label{sec:main4}
In this section, our stability analysis is based on a linearized model to avoid complications arising from the original nonlinear problem. We assume that the flow is governed by the Stokes equations and that both the fluids and the Biot regions remain fixed. These assumptions are common in the analysis of domain decomposition methods for fluid-structure interaction (FSI) problems \cite{MartinaOyekole, seboldt2021numerical}. We will therefore skip all hat’s that indicate reference variables. The time-splitting scheme in the weak formulation is presented in Algorithm \ref{fixedAL}.

For simplicity of notation, let $\Vert\cdot\Vert_f$ denote $\Vert\cdot\Vert_{L^2(\Omega_f)}$, $\Vert\cdot\Vert_p$ denote $\Vert\cdot\Vert_{L^2(\Omega_p)}$, $\Vert\cdot\Vert_\Gamma$ denote $\Vert\cdot\Vert_{L^2(\Gamma)}$ and $\Vert\cdot\Vert_N$ denote $\Vert\cdot\Vert_{L^2(\Sigma_f^N)}$. In the following, we introduce the elastic energy of the solid defined as:
\begin{equation}
    \Vert\boldsymbol{\eta}^n\Vert^2_S=2\mu_p\Vert\mathbb D(\boldsymbol{\eta}^n)\Vert^2_p+\lambda_p\Vert\nabla\cdot\boldsymbol{\eta}^n\Vert^2_p,
\label{energynorm}
\end{equation}
and the sum of kinetic and elastic energy of the solid, along with the kinetic energy of the fluid, given by:
$$
    \mathcal E^n=\frac{\rho_p}{2}\Vert\boldsymbol{\xi}^n\Vert^2_p+\frac{1}{2}\Vert\boldsymbol{\eta}^n\Vert_S^2+\frac{C_0}{2}\Vert\phi^n\Vert^2_p+\frac{\rho_f}{2}\Vert\boldsymbol{u}^n\Vert^2_f.
$$
Let $\mathcal D^n$ denote the fluid dissipation term:
$$
\mathcal D^n=\mu_f\Delta t\Vert\mathbb D(\boldsymbol{u}^n)\Vert^2_f + \frac{\Delta t}{2}K_m\Vert\nabla\phi^n\Vert^2_p,
$$
$\mathcal I^n$ denote the energy on the interface $\Gamma$:
\begin{align*}
    \mathcal I^n=\frac{\Delta t}{2}\left(
    L\Vert\boldsymbol u^{n}\cdot\boldsymbol{n}_f\Vert^2_\Gamma + \frac{\Vert\phi^n\Vert^2_\Gamma}{L} +\Vert\boldsymbol{\xi}^n\cdot\boldsymbol{n}_p\Vert_\Gamma^2+\gamma\sum_{j=1}^{d-1}\Vert\boldsymbol u^{n}\cdot\boldsymbol{\tau}_{f,j}\Vert^2_\Gamma+\Vert\boldsymbol \xi^{n}\cdot\boldsymbol{\tau}_{p,j}\Vert^2_\Gamma\right)
    \end{align*}
$\mathcal{N}^n$ represent the remaining terms due to numerical dissipation:
\begin{align*}
    &\mathcal N^n=\frac{\rho_p}{2}\sum_{i=1}^n\Vert\boldsymbol{\xi}^i-\boldsymbol{\xi}^{i-1}\Vert^2_p + \mu_p\sum_{i=1}^n\Vert\mathbb D(\boldsymbol{\eta}^i-\boldsymbol{\eta}^{i-1})\Vert^2_p
    +\frac{\lambda_p}{2}\sum_{i=1}^n\Vert\nabla\cdot(\boldsymbol{\eta}^i-\boldsymbol{\eta}^{i-1})\Vert^2_p\\
    &+\frac{\rho_f}{2}\sum_{i=1}^n\Vert\boldsymbol{u}^i-\boldsymbol{u}^{i-1}\Vert^2_f
    +\frac{C_0}{2}\sum_{i=1}^n\Vert\phi^i-\phi^{i-1}\Vert^2_p+\frac{\Delta t}{2}\gamma\sum_{j=1}^{d-1}\sum_{i=1}^n\Vert(\boldsymbol{u}^i-\boldsymbol{\xi}^{i-1})\cdot\boldsymbol{\tau}_{f, j}\Vert^2_\Gamma\\
    &+\frac{\Delta t}{2}\gamma\sum_{j=1}^{d-1}\sum_{i=1}^n\Vert(\boldsymbol{\xi}^i-\boldsymbol{u}^{i-1})\cdot\boldsymbol{\tau}_{p, j}\Vert^2_\Gamma+\frac{\Delta t}{2}\sum_{i=1}^{n}\Vert\left(\boldsymbol{\xi}^{i}-\boldsymbol{\xi}^{i-1}\right)\cdot\boldsymbol{n}_p\Vert^2_\Gamma,
\end{align*}
and $\mathcal F^n$ be the energy of the forcing terms:
$$
    \mathcal{F}^n = \Delta t\frac{C_{p1}^2C_K^2}{2\mu_f}\Vert\boldsymbol{F}_f^{n+1}\Vert^2_f + \Delta t\frac{C_{tr}^2C_K^2C_{p1}}{2\mu_f}\Vert\boldsymbol{g}^{n+1}\Vert^2_N + \Delta t\frac{C_{p2}^2}{2K_m}\Vert F_d^{n+1}\Vert^2_p,
$$
where $C_{p1}, C_{p2}$ are constants related to Poincaré-Friedrichs inequality, while $C_K, C_{tr}$ are constants associated with Korn's inequality and the trace inequality, respectively.
In addition, the polarization identity will be applied
\begin{equation}
\label{polar}
    2(a-b)a = a^2 - b^2 + (a-b)^2.
\end{equation}
Let $T=n\Delta t$. The energy estimate for the proposed decoupled scheme is given in the following theorem.
\begin{theorem}
Let $(\boldsymbol{u}^n, p^n, \boldsymbol{\xi}^n, \boldsymbol{\eta}^n, \phi^n)$ be the solution of the decoupled scheme described by Algorithm \ref{fixedAL}. Then, this scheme is unconditionally stable and a priori energy estimate holds as follows:
\begin{equation}
\mathcal E^n+\mathcal D^n + \mathcal I^n + \mathcal N^n \le \left(\mathcal F^n + \mathcal E^0 + \mathcal D^0 + \mathcal I^0\right)e^{CT},
                    \label{priori}
\end{equation}
where $C$ is a positive constant deriving from the Gronwall's inequality.
\end{theorem}
\begin{proof}
By taking $(\boldsymbol{v}, q, \boldsymbol{\zeta}, \psi)=\Delta t(\boldsymbol{u}^{n+1}, p^{n+1}, \boldsymbol{\xi}^{n+1}, \phi^{n+1})$ in (\ref{Stable1}) and (\ref{Stable2}), adding the equations together, integrating by parts and applying (\ref{polar}), one yields:
\begin{align}
        &\frac{\rho_f}{2}\left(\Vert\boldsymbol{u}^{n+1}\Vert^2_f-\Vert\boldsymbol{u}^n\Vert^2_f+\Vert\boldsymbol{u}^{n+1}-\boldsymbol{u}^n\Vert^2_f\right) + 2\mu_f\Delta t\Vert\mathbb D(u^{n+1})\Vert^2_f + L\Delta t\Vert u^{n+1}\cdot\boldsymbol{n}_f\Vert_\Gamma^2\\
        &+\frac{\rho_p}{2}\left(\Vert\boldsymbol{\xi}^{n+1}\Vert^2_p-\Vert\boldsymbol{\xi}^n\Vert^2_p+\Vert\boldsymbol{\xi}^{n+1}-\boldsymbol{\xi}^n\Vert^2_p\right) +
        \mu_p\left(\Vert\mathbb D(\boldsymbol{\eta}^{n+1})\Vert^2_p
        -\Vert\mathbb D(\boldsymbol{\eta}^n)\Vert^2_p\right. \nonumber\\
        &\left.+\Vert\mathbb D(\boldsymbol{\eta}^{n+1}-\boldsymbol{\eta}^n)\Vert^2_p\right)
        +\frac{\lambda_p}{2}\left(\Vert\nabla\cdot\boldsymbol{\eta}^{n+1}\Vert^2_p
         -\Vert\nabla\cdot\boldsymbol{\eta}^n\Vert^2_p+\Vert\nabla\cdot(\boldsymbol{\eta}^{n+1}-\boldsymbol{\eta}^n)\Vert^2_p\right) \nonumber\\
         &+\Delta t\Vert\boldsymbol{\xi}^{n+1}\cdot\boldsymbol{n}_p\Vert_\Gamma^2+\frac{C_0}{2}\left(\Vert\phi^{n+1}\Vert^2_p-\Vert\phi^n\Vert^2_p+\Vert\phi^{n+1}-\phi^n\Vert^2_p\right)
        +\Delta t\Vert\mathbb K^{\frac{1}{2}}\nabla\phi^{n+1}\Vert^2_p\nonumber\\
        &+\frac{\Delta t}{L}\Vert\phi^{n+1}\Vert^2_\Gamma
        =\Delta t(\boldsymbol{F}^{n+1}_f, \boldsymbol{u}^{n+1})_f
        +\Delta t(F^{n+1}_d, \phi^{n+1})_p
        +\Delta t\langle\boldsymbol{g}^{n+1},\boldsymbol{u}^{n+1}\rangle_{\Sigma_f^N}\nonumber\\
        &+ \Delta t\langle R_1^n, \boldsymbol{u}^{n+1}\cdot\boldsymbol{n}_f\rangle_\Gamma
        + \Delta t\langle R_3^n, \boldsymbol{\xi}^{n+1}\cdot\boldsymbol{n}_p\rangle_\Gamma
        + \Delta t\langle R_4^n, \phi^{n+1}\rangle_\Gamma\nonumber\\
        &+\Delta t\gamma\langle P_p(\boldsymbol{u}^n-\boldsymbol{\xi}^{n+1}), P_p\boldsymbol{\xi}^{n+1}\rangle_\Gamma - \Delta t\gamma\langle P_f(\boldsymbol{u}^{n+1}-\boldsymbol{\xi}^n), P_f\boldsymbol{u}^{n+1}\rangle_\Gamma\nonumber.
\end{align}
For the left-hand side, we replace the last term into $\Delta tK_m\Vert\nabla\phi^{n+1}\Vert^2$ in the following derivation because of (\ref{uniformly}).
We bound the right-hand side as follows. Applying the polarization identity (\ref{polar}), one yields the term containing $R_1^n$ as
\begin{align*}
    &\quad\ \Delta t\langle R_1^n, \boldsymbol{u}^{n+1}\cdot\boldsymbol{n}_f\rangle_\Gamma\\
    &= \Delta t\langle L\boldsymbol{u}^n\cdot\boldsymbol{n}_f-\phi^n,\boldsymbol{u}^{n+1}\cdot\boldsymbol{n}_f\rangle_\Gamma\\
    &=\Delta tL\langle \boldsymbol{u}^n\cdot\boldsymbol{n}_f,\boldsymbol{u}^n\cdot\boldsymbol{n}_f+(\boldsymbol{u}^{n+1}-\boldsymbol{u}^n)\cdot\boldsymbol{n}_f\rangle_\Gamma - \Delta t\langle\phi^n, \boldsymbol{u}^{n+1}\cdot\boldsymbol{n}_f\rangle_\Gamma\\
    &=\frac{\Delta tL}{2}\left(\Vert\boldsymbol{u}^n\cdot\boldsymbol{n}_f\Vert_\Gamma^2+\Vert\boldsymbol{u}^{n+1}\cdot\boldsymbol{n}_f\Vert_\Gamma^2-\Vert(\boldsymbol{u}^{n+1}-\boldsymbol{u}^n)\cdot\boldsymbol{n}_f\Vert^2_\Gamma\right)-\Delta t\langle\phi^n, \boldsymbol{u}^{n+1}\cdot\boldsymbol{n}_f\rangle_\Gamma.
\end{align*}
Similarly, terms of $R_3^n$ and $R_4^n$ can be rewritten as
\begin{align*}
   &\quad\ \Delta t\langle R_3^n, \boldsymbol{\xi}^{n+1}\cdot\boldsymbol{n}_p\rangle_\Gamma=\frac{\Delta t}{2}\left(\Vert\boldsymbol{\xi}^n\cdot\boldsymbol{n}_f\Vert_\Gamma^2+\Vert\boldsymbol{\xi}^{n+1}\cdot\boldsymbol{n}_f\Vert_\Gamma^2-\Vert(\boldsymbol{\xi}^{n+1}-\boldsymbol{\xi}^n)\cdot\boldsymbol{n}_f\Vert^2_\Gamma\right) 
\end{align*}
and
\begin{align*}
    \Delta t\langle R_4^n, \phi^{n+1}\cdot\boldsymbol{n}_f\rangle_\Gamma 
    =\frac{\Delta t}{2L}\left(\Vert\phi^n\Vert_\Gamma^2+\Vert\phi^{n+1}\Vert_\Gamma^2-\Vert\phi^{n+1}-\phi^n\Vert^2_\Gamma\right)+\Delta t\langle\boldsymbol{u}^n\cdot\boldsymbol{n}_f, \phi^{n+1}\rangle_\Gamma.
\end{align*}
To deal with the residual term $\Delta t\langle\boldsymbol{u}^n\cdot\boldsymbol{n}_f, \phi^{n+1}\rangle_\Gamma - \Delta t\langle\phi^n, \boldsymbol{u}^{n+1}\cdot\boldsymbol{n}_f\rangle_\Gamma$ involving asynchronous values, we first use the Cauchy-Schwarz and Young's inequality as follows:
\begin{align*}
    &\Delta t\langle\boldsymbol{u}^n\cdot\boldsymbol{n}_f, \phi^{n+1}\rangle_\Gamma - \Delta t\langle\phi^n, \boldsymbol{u}^{n+1}\cdot\boldsymbol{n}_f\rangle_\Gamma\\
    =&\Delta t\langle\boldsymbol{u}^n\cdot\boldsymbol{n}_f, \phi^{n+1}-\phi^n\rangle_\Gamma + \Delta t\langle\phi^n, (\boldsymbol{u}^n-\boldsymbol{u}^{n+1})\cdot\boldsymbol{n}_f\rangle_\Gamma\\
    \le&\frac{\Delta t}{2L}\Vert\phi^{n+1}-\phi^n\Vert^2_\Gamma+\frac{\Delta tL}{2}\Vert(\boldsymbol{u}^{n+1}-\boldsymbol{u}^n)\cdot\boldsymbol{n}_f\Vert^2_\Gamma+\frac{\Delta tL}{2}\Vert\boldsymbol{u}^n\cdot\boldsymbol{n}_f\Vert^2_\Gamma+\frac{\Delta t}{2L}\Vert\phi^n\Vert_\Gamma^2.
\end{align*}
Note that the first two terms above will both be canceled out and what remain are the last two terms. To bound them, applying the trace, Korn's and Young's inequality, we have
\begin{align*}
    &\frac{\Delta tL}{2}\Vert\boldsymbol{u}^n\cdot\boldsymbol{n}_f\Vert^2_\Gamma+\frac{\Delta t}{2L}\Vert\phi^n\Vert_\Gamma^2\le\frac{\Delta t C_{tr1}^2}{2L}\Vert\nabla\phi^n\Vert_p\Vert\phi^n\Vert_p+\frac{\Delta t}{2}LC_{tr2}^2C_K\Vert\mathbb D(\boldsymbol{u}^n)\Vert_f\Vert\boldsymbol{u}^n\Vert_f\\
&\le\frac{\Delta tK_m}{2}\Vert\nabla\phi^n\Vert^2_p+\frac{\Delta tC_{tr1}^4}{8K_mL^2}\Vert\phi^n\Vert^2_p+\mu_f\Delta t\Vert\mathbb D(\boldsymbol{u}^n)\Vert^2_f+\frac{C_{tr2}^4C_K^2L^2}{16\mu_f}\Delta t\Vert\boldsymbol{u}^n\Vert^2_f.
\end{align*}
We rewrite the tangential components as follows
\begin{align*}
&\Delta t\gamma\langle P_p(\boldsymbol{u}^n-\boldsymbol{\xi}^{n+1}), P_p\boldsymbol{\xi}^{n+1}\rangle - \Delta t\gamma\langle P_f(\boldsymbol{u}^{n+1}-\boldsymbol{\xi}^n), P_f\boldsymbol{u}^{n+1}\rangle\\
=&\frac{\Delta t\gamma}{2}\sum_{j=1}^{d-1}\left(\Vert\boldsymbol{u}^n\cdot\boldsymbol{\tau}_{p, j}\Vert^2_\Gamma-\Vert\boldsymbol{\xi}^{n+1}\cdot\boldsymbol{\tau}_{p, j}\Vert^2_\Gamma-\Vert(\boldsymbol{u}^n-\boldsymbol{\xi}^{n+1})\cdot\boldsymbol{\tau}_{p, j}\Vert^2_\Gamma\right)\\
-&\frac{\Delta t\gamma}{2}\sum_{j=1}^{d-1}\left(\Vert\boldsymbol{u}^{n+1}\cdot\boldsymbol{\tau}_{f, j}\Vert^2_\Gamma-\Vert\boldsymbol{\xi}^n\cdot\boldsymbol{\tau}_{f, j}\Vert^2_\Gamma+\Vert(\boldsymbol{u}^{n+1}-\boldsymbol{\xi}^n)\cdot\boldsymbol{\tau}_{f, j}\Vert^2_\Gamma\right).
\end{align*}
By using the the Cauchy-Schwarz, Poincaré-Friedrichs, trace, Korn's and Young's inequality, we bound the forcing terms as:
\begin{align*}
    &\Delta t(\boldsymbol{F}_f^{n+1}, \boldsymbol{u}^{n+1})+\Delta t\langle\boldsymbol{g}^{n+1}, \boldsymbol{u}^{n+1}\rangle_N+\Delta t(F^{n+1}_d, \phi^{n+1})\\
    \le&\Delta t\frac{C_{p1}^2C_K^2}{2\mu_f}\Vert \boldsymbol{F}_f^{n+1}\Vert^2+\frac{\Delta t}{2}\mu_f\Vert\mathbb D(\boldsymbol{u}^{n+1})\Vert^2+\Delta t\frac{C_{tr3}^2C_K^2C_{p1}}{2\mu_f}\Vert \boldsymbol{g}^{n+1}\Vert^2_N\\
    &+\frac{\Delta t}{2}\mu_f\Vert\mathbb D(\boldsymbol{u}^{n+1})\Vert^2
    +\Delta t\frac{C_{p2}^2}{2K_m}\Vert F_d^{n+1}\Vert^2+\Delta t\frac{K_m}{2}\Vert\nabla\phi^{n+1}\Vert^2.
\end{align*}
Finally, combining all estimations above and summing from $0$ to $n-1$, we get:
$$
\mathcal E^n+\mathcal D^n + \mathcal I^n + \mathcal N^n \le \mathcal F^n + \mathcal E^0 + \mathcal D^0 + \mathcal I^0 + \frac{C_{tr1}^4}{8K_mL^2}\Delta t\sum_{i=0}^{n-1}\Vert\phi^i\Vert^2 + \frac{C_{tr2}^4C_K^2L^2}{16\mu_f}\Delta t\sum_{i=0}^{n-1}\Vert \boldsymbol{u}^i\Vert^2.
$$
By using the Gronwall's inequality, the last two terms above can be bounded and one yields the priori energy estimate (\ref{priori}).
\end{proof}

\begin{algorithm}[H]
  \caption{ Fully parallelizable splitting scheme for linearized Stokes-Biot problem}  
  \label{fixedAL}
  Given $\boldsymbol{u}^n, \boldsymbol{\eta}^n, \boldsymbol{\xi}^n$ and $\phi^n$, we solve for $\boldsymbol{u}^{n+1}, \boldsymbol{\eta}^{n+1}, \boldsymbol{\xi}^{n+1}$ and $\phi^{n+1}$, for $n=0, 1, 2, \cdots, N-1.$
  
\begin{algorithmic}[H]
\STATE \textbf{Set}
    \begin{align*}
        R_{1}^n=&L\boldsymbol{u}^n\cdot\boldsymbol{n}_f^n-\phi^n,
        && R_{2}^n =\gamma P_f^n\left(\boldsymbol{\xi}^n\right),\\
        R_{3}^n=&\boldsymbol{\xi}^n\cdot\boldsymbol{n}_p,
        &&R_{4}^n=-\boldsymbol{u}^n\cdot\boldsymbol{n}_p+\phi^n/L,\\
        R_{5}^n =& \gamma P_p\left(\boldsymbol{u}^n\right).
    \end{align*}
    
\STATE \textbf{Solve the following two sub-problems in parallel.}

\vspace{0.2in}
\noindent  \underline{\textbf{Fluid Subproblem:}}
Solve (\ref{Stable1}) for $\boldsymbol{u}^{n+1}, p^{n+1}$.
 \boxedeq{
     \begin{aligned}
        &\rho_f\left(\frac{\boldsymbol{u}^{n+1}-\boldsymbol{u}^n}{\Delta t}, \boldsymbol{v}\right)_f 
        + 2\mu_f\left(\mathbb D(\boldsymbol{u}^{n+1}), \mathbb{D}(\boldsymbol{v})\right)_f
        -(p^{n+1}, \nabla\cdot\boldsymbol{v})_f
        +(\nabla\cdot\boldsymbol{u}^{n+1}, q)_f\\
        &+\gamma\langle P_f\boldsymbol{u}^{n+1}, P_f\boldsymbol{v}\rangle_{\Gamma}
        +L\langle\boldsymbol{u}^{n+1}\cdot\boldsymbol{n}_f, \boldsymbol{v}\cdot\boldsymbol{n}_f\rangle_{\Gamma} 
        = \langle R_{2}^{n}, P_f\boldsymbol{v}\rangle_{\Gamma}
        +\left(\boldsymbol{F}^{n+1}_f, \boldsymbol{v}\right)_f\\
        &+\langle R_1^n, \boldsymbol{v}\cdot\boldsymbol{n}_f\rangle_{\Gamma}
        +\langle \boldsymbol{g}^{n+1}, \boldsymbol{v}\rangle_{\Sigma_f^N}, 
        ~\forall (\boldsymbol{v}, q)\in (V_f, Q_f)
    \end{aligned}
    \label{Stable1}
}

\noindent  \underline{\textbf{Structure Subproblem:}}
Solve (\ref{Stable2}) for $\partial_t\boldsymbol{\eta}^{n+1}=\boldsymbol{\xi}^{n+1}$ and $\phi^{n+1}$.
 \boxedeq{
     \begin{aligned}
      \label{Stable2}
        &\rho_p\left(\frac{\boldsymbol{\xi}^{n+1}-\boldsymbol{\xi}^n}{\Delta t}, \boldsymbol{\zeta}\right)_p 
        + 2\mu_p\left(\mathbb D(\boldsymbol{\eta}^{n+1}), \mathbb{D}(\boldsymbol{\zeta})\right)_p + \lambda_p(\nabla\cdot\boldsymbol{\eta}^{n+1}, \nabla\cdot\boldsymbol{\zeta})_p\\
        &-\alpha(\phi^{n+1}, \nabla\cdot\boldsymbol{\zeta})_p
        +C_0\left(\frac{\phi^{n+1}-\phi^n}{\Delta t}, \psi\right)_p+\alpha(\nabla\cdot\boldsymbol{\xi}^{n+1}, \psi)_p+(\mathbb K\nabla\phi^{n+1},\nabla\psi)_p\\
        &+\gamma\langle P_p\boldsymbol{\xi}^{n+1}, P_p\boldsymbol{\zeta}\rangle_{\Gamma}
        +\langle\boldsymbol{\xi}^{n+1}\cdot\boldsymbol{n}_p, \boldsymbol{\zeta}\cdot\boldsymbol{n}_p\rangle_{\Gamma}
        +\langle\phi^{n+1}, \boldsymbol{\zeta}\cdot\boldsymbol{n}_p\rangle_{\Gamma}
        +\frac{1}{L}\langle\phi^{n+1},\psi\rangle_{\Gamma}\\
        &-\langle\boldsymbol{\xi}^{n+1}\cdot\boldsymbol{n}_p,\psi\rangle_{\Gamma}
        =\langle  R_{5}^{n}, P_p\boldsymbol{\zeta}\rangle_{\Gamma}
        +\left(F^{n+1}_d,\psi\right)_p
        +\langle R_2^n, \boldsymbol{\zeta}\cdot\boldsymbol{n}_p\rangle_{\Gamma}
        +\langle R_3^n, \psi\rangle_{\Gamma}, \\
        &\forall (\boldsymbol{\zeta}, \psi)\in (V_p, Q_p)
    \end{aligned}
    }
    
\vspace{0.1in} 
\STATE \textbf{Advancing to the next timestep.}
\end{algorithmic}
\end{algorithm}

\section{Numerical examples}
\label{sec:main5}
In this section, we present a few numerical examples to demonstrate the stability and accuracy of our proposed method. We use the finite element method for the spatial discretization, employing Taylor-Hood $P2-P1$ elements for fluid variables. In terms of Biot variables, some researchers found that the standard finite element method is unstable for two-field (displacement-pressure) Biot formulations, particularly when dealing with very small $\mathbb K$ and $C_0$, which may lead to pressure oscillations, or large $\lambda_p$, being associated with locking of displacement\cite{murad1992improved, boon2023mixed, rodrigo2018new}. In this regard, we use $P1$ element for Darcy pressure and $P2$ element for displacement to construct the mixed stable finite element pairs. The Code for all numerical examples reported below is written in FEniCS \cite{AlnaesEtal2014, AlnaesEtal2015}, while the parallelization for solving two subproblems independently is implemented through \textit{multiprocessing} library in Python.

\begin{remark} 
We choose $\mathbb{K} = \mathcal{K}\boldsymbol{I}$ in the following numerical examples for simplicity, where $\mathcal{K}$ is a positive constant. In real physical problems, the value of $\mathcal{K}$ is usually much less than 1. For the choice of parameter $L$ in the Robin conditions, we suggest taking $L = \frac{1}{\mathcal{K}}$ from a numerical perspective. To illustrate how this choice works, we rewrite the Robin-type condition for Darcy pressure $\phi$ as:
$$
R_3=\mathcal{K}\nabla\phi^{n+1}\cdot\boldsymbol{n}_p+\frac{1}{L}\phi^{n+1}=\boldsymbol{\xi}^{n+1}\cdot\boldsymbol{n}_p-\boldsymbol{u}^{n}\cdot\boldsymbol{n}_p+\frac{1}{L}\phi^{n}.
$$
If $\mathcal{K}$ is too small compared to $\frac{1}{L}$, then the term $\frac{1}{L}\phi^{n+1} = \frac{1}{L}\phi^n$ becomes predominant, while the interface condition 
$\mathcal{K}\nabla\phi\cdot\boldsymbol{n}_p \approx \boldsymbol{\xi}\cdot\boldsymbol{n}_p - \boldsymbol{u}\cdot\boldsymbol{n}_p$ has little effect. Therefore, to keep all variables at a comparable level, we take $L = \frac{1}{\mathcal{K}}$ in the following computations.
\end{remark}

\subsection{Benchmark problem}
\label{real}
In this example, we considered two benchmark problems with manufactured solutions \cite{seboldt2021numerical} to examine the rate of convergence and the temporal accuracy of the proposed scheme. We solved the time-dependent Stoke-Biot system with added external forcing terms. The system is given as follows:
$$
\left\{
\begin{array}{ll}
    \rho_f \partial_t \boldsymbol{u}=\nabla \cdot \boldsymbol{\sigma}_f\left(\boldsymbol{u}, p\right)+\boldsymbol{F}_f & \text { in } \Omega_f \times(0, T), \\
    \nabla \cdot \boldsymbol{u}=g_f & \text { in } \Omega_f \times(0, T), \\
    \partial_t \boldsymbol{\eta}=\boldsymbol{\xi} & \text { in } \Omega_p \times(0, T), \\
    \rho_p \partial_t \boldsymbol{\xi}=\nabla \cdot \boldsymbol{\sigma}_p\left(\boldsymbol{\eta}, \phi\right)+\boldsymbol{F}_e & \text { in } \Omega_p \times(0, T), \\
    \boldsymbol{u_p}=-\mathbb K\nabla \phi & \text { in } \Omega_p \times(0, T), \\
    C_0 \partial_t \phi+\alpha \nabla \cdot \boldsymbol{\xi}-\nabla \cdot (\mathbb K\nabla\phi)=F_d & \text { in } \Omega_p \times(0, T) .
\end{array}
\right.
$$
The FPSI problem is defined within a rectangular domain where the fluid domain occupies the upper half, $\Omega_f = (0,1) \times (0,1)$, and the solid domain occupies the lower half, $\Omega_p = (0,1) \times (-1,0)$. The following physical parameters are used: $\rho_p=$ $\mu_p=\lambda_p=\alpha=C_0=\gamma=\rho_f=\mu_f=1$, and $\mathbb K=\boldsymbol{I}$. The final time is $T=1$. The exact solutions for case 1 are given by:
\begin{align*}
& \boldsymbol{\eta}_{1}=\sin (\pi t)\left[\begin{array}{c}
-3 x+\cos (y) \\
y+1
\end{array}\right], \quad&&\phi_{1}=e^t \sin (\pi x) \cos \left(\frac{\pi y}{2}\right), \\
& \boldsymbol{u}_{1}=\pi \cos (\pi t)\left[\begin{array}{c}
-3 x+\cos (y) \\
y+1
\end{array}\right], \quad&&p_{1}=e^t \sin (\pi x) \cos \left(\frac{\pi y}{2}\right)+2 \pi \cos (\pi t) .
\end{align*}
And the exact solutions for case 2 are:
\begin{align*}
& \boldsymbol{\eta}_{2}=\sin (\pi t)\left[\begin{array}{c}
-3 x+\cos (y) \\
y+1
\end{array}\right], &&\phi_{2}=\sin(\pi t+\frac{\pi}{4}) \sin (\pi x) \cos \left(\frac{\pi y}{2}\right), \\
& \boldsymbol{u}_{2}=\pi \cos (\pi t)\left[\begin{array}{c}
-3 x+\cos (y) \\
y+1
\end{array}\right], &&p_{2}=\sin(\pi t+\frac{\pi}{4})\sin (\pi x) \cos \left(\frac{\pi y}{2}\right)+2 \pi \cos (\pi t) .
\end{align*}
From the exact solutions, we can retrieve the forcing terms of $\boldsymbol{F}_f, g_f, \boldsymbol{F}_e$, and $F_d$. Regarding the boundary conditions, Neumann boundary conditions are enforced on the right side of the fluid domain and on the bottom of the Biot domain for $\phi$, while Dirichlet conditions are used on all other external boundaries. The simulation continued until the final time $T=1$, after which we assessed the numerical error. To compute the rates of convergence, we define the errors for the structure displacement and velocity, the Darcy pressure, and the fluid velocity, respectively, as:
\begin{align*}
e_{i,\boldsymbol{\eta}}=\Vert&\boldsymbol{\eta}-\boldsymbol{\eta}_{i}\Vert_S, \quad e_{i,\boldsymbol{\xi}}=\Vert\boldsymbol{\xi}-\boldsymbol{\xi}_{i}\Vert_{L^2\left(\Omega_p\right)},  \quad e_{i,\phi}=\Vert\phi-\phi_{i}\Vert_{L^2\left(\Omega_p\right)},\\
&\quad e_{i,\boldsymbol{u}}=\Vert\boldsymbol{u}-\boldsymbol{u}_{i}\Vert_{L^2\left(\Omega_f\right)}, \quad \quad e_{i,p}=\Vert p-p_{i}\Vert_{L^2\left(\Omega_f\right)},
\end{align*}
where $i=1$ for Case 1 and 2 for Case 2. To compute the rates of convergence, we choose the following time and space discretization parameters for different choices of $n$ from 4 to 128:
$$
\left\{\Delta t, \Delta x\right\} = \left\{\frac{0.05}{n}, \frac{0.5}{n}\right\},
$$
where $\Delta x$ is the mesh size. Error information and corresponding convergent rate are given in Table \ref{convergence1}, Table \ref{convergence2} and Figure \ref{fig:Error}. We see that our partitioned method reaches the first order accuracy in time for both Stokes and Biot variables without iteration for both Case 1 and Case 2.
\begin{table}[h!]
    \caption{Errors of the Stokes-Biot partitioned algorithm for Case 1 at final time $T=1.0$.}
    \centering
    \begin{tabular}{c|c|c|c|c|c}
    \hline
        $n$ & $e_{1,\boldsymbol{\eta}}$ & $e_{1,\boldsymbol{\xi}}$ & $e_{1,\phi}$ & $e_{1,\boldsymbol{u}}$ & $e_{1,p}$\\
        \hline
        \hline
        4 & 1.34E-01 & 1.28E-01 & 2.42E-02 & 1.34E-02 & 1.75E-01\\
        8 & 6.63E-02 & 6.49E-02 & 5.77E-03 & 6.84E-03 & 8.98E-02\\
       16 & 3.31E-02 & 3.26E-02 & 2.47E-03 & 3.46E-03 & 4.55E-02\\
       32 & 1.65E-02 & 1.64E-02 & 1.22E-03 & 1.74E-03 & 2.29E-02\\
       64 & 8.27E-03 & 8.21E-03 & 6.18E-04 & 8.75E-04 & 1.15E-02\\
      128 & 4.14E-03 & 4.11E-03 & 3.13E-04 & 4.38E-04 & 5.76E-03\\
 \hline
    \end{tabular}
    \label{convergence1}
\end{table}

\begin{table}[h!]
    \caption{Errors of the Stokes-Biot partitioned algorithm for Case 2 at final time $T=1.0$.}
    \centering
    \begin{tabular}{c|c|c|c|c|c}
    \hline
        $n$ & $e_{2,\boldsymbol{\eta}}$ & $e_{2,\boldsymbol{\xi}}$ & $e_{2,\phi}$ & $e_{2,\boldsymbol{u}}$ & $e_{2,p}$\\
        \hline
        \hline
       4 & 1.66E-01 & 1.25E-01 & 1.57E-02 & 1.41E-02 & 2.12E-01\\
       8 & 8.49E-02 & 6.36E-02 & 6.60E-03 & 7.24E-03 & 1.06E-01\\
      16 & 4.29E-02 & 3.21E-02 & 3.12E-03 & 3.67E-03 & 5.32E-02\\
      32 & 2.16E-02 & 1.61E-02 & 1.53E-03 & 1.85E-03 & 2.66E-02\\
      64 & 1.08E-02 & 8.08E-03 & 7.56E-04 & 9.29E-04 & 1.33E-02\\
     128 & 5.43E-03 & 4.05E-03 & 3.76E-04 & 4.65E-04 & 6.66E-03\\
 \hline
    \end{tabular}
    
    \label{convergence2}
\end{table}
\begin{figure}[h!]
  \centering
  \includegraphics[scale=0.052]{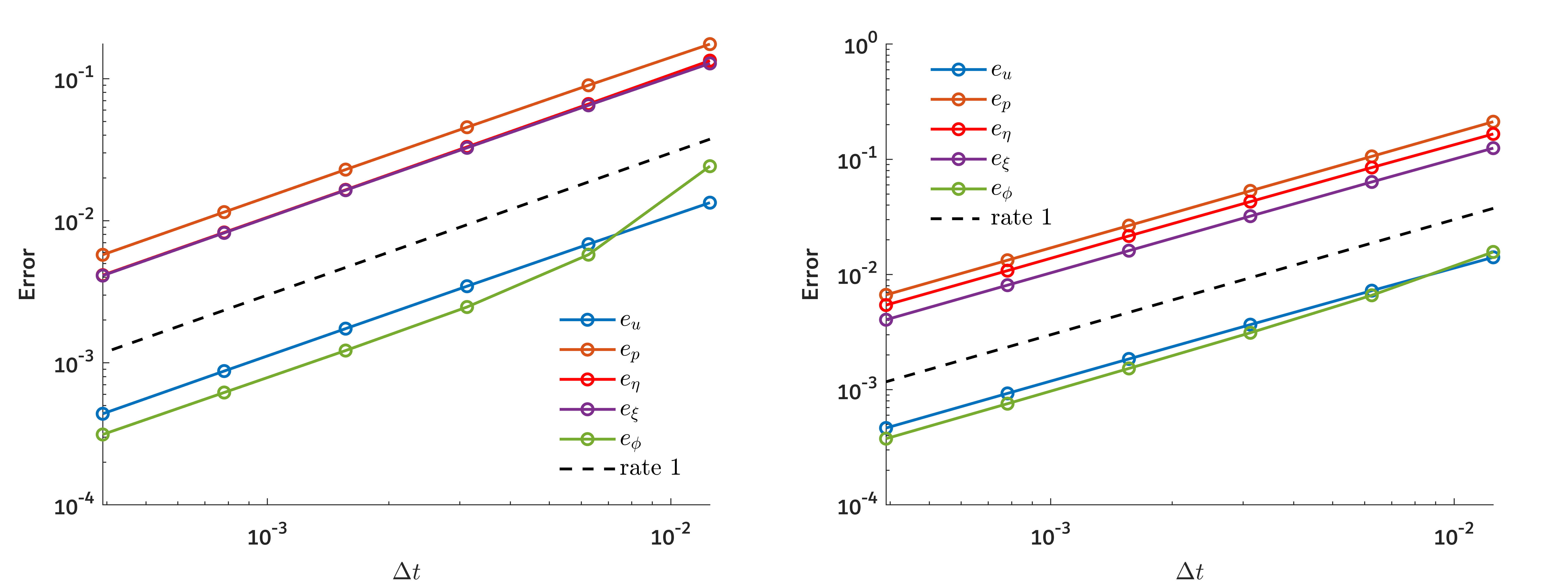}
  \caption{Convergence rate in time of the proposed decoupled algorithm for Case 1 (left) and Case 2 (right) at final time $T=1.0$.}
  \label{fig:Error}
\end{figure}

\subsection{2D fluid flow through a channel with poroelastic obstacles}
\label{2dexample}
In this section, we consider a $2D$ NS channel flow featuring two poroelastic obstacles on either side of the channel near the inlet. A sketch of this geometry is shown in Figure \ref{fig:domain_2d}. We impose no-slip boundary condition for upper and lower sides channel walls, while a natural outflow condition, $\boldsymbol{\sigma}_f\boldsymbol{n}_f=0$, is used for the outlet. 
For top and bottom boundaries of Biot domains, we set $\boldsymbol{\hat{\eta}}=\boldsymbol{\hat{\xi}}=\boldsymbol{0}$ and $\hat{\nabla}\hat{\phi}\cdot\boldsymbol{\hat{n}}_p=\boldsymbol{0}$ for no penetration (no-flux) conditions. This geometry can be perceived as a narrowing artery with stenosis or a narrow tunnel in intertidal zone, formed by the growth of Spartina. 
\begin{figure}[h!]
    \centering
    \includegraphics[width=1.0\linewidth]{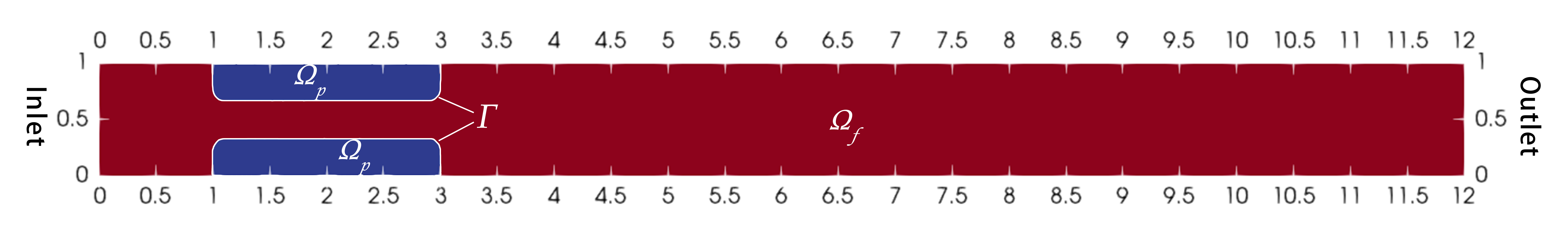}
    \caption{An illustration of the $2D$ computation domain, where the region in red $\Omega_f$ stands for the fluid region, and regions in blue $\Omega_p$ denotes the poroelastic region. The size of the computational domain is $1.0\ cm \times 12.0\ cm$.}
    \label{fig:domain_2d}
\end{figure}

The parameters used for fluids simulations are as follows: the kinematic viscosity is set to $\nu=0.01\text{ cm}^2$/s and density (per $2D$ unit area) is $\rho_f=1.0$ g/c$\text{m}^2$. For poroelastic structures, we took density to be $\rho_p=1.2$ g/c$\text{m}^2$, and two Lame's parameters to be $\mu_p=1.0336\times10^3$ dyne/cm and $\lambda_p=4.9364\times10^4$ dyne/cm, respectively. The pressure storage coefficient is taken as $C_0 = 10^{-3}~ cm/dyne$, and the permeability $\mathcal{K}=10^{-3}~cm^2s/g$. The slip constant is $\gamma=1/\sqrt{\mathcal{K}}$ $ 
g/(cm \cdot s)$. We also set the Biot-Willis parameter $\alpha$ to 1. This problem was solved with $\Delta t=10^{-3}$ s, and the final time $T=5$ s. Furthermore, we prescribe the inlet velocity to be a parabolic profile $\boldsymbol{u}=(20x(1-x), 0)$ in horizontal direction.

\begin{figure}[htb]
    \centering
    \includegraphics[width=1.0\linewidth]{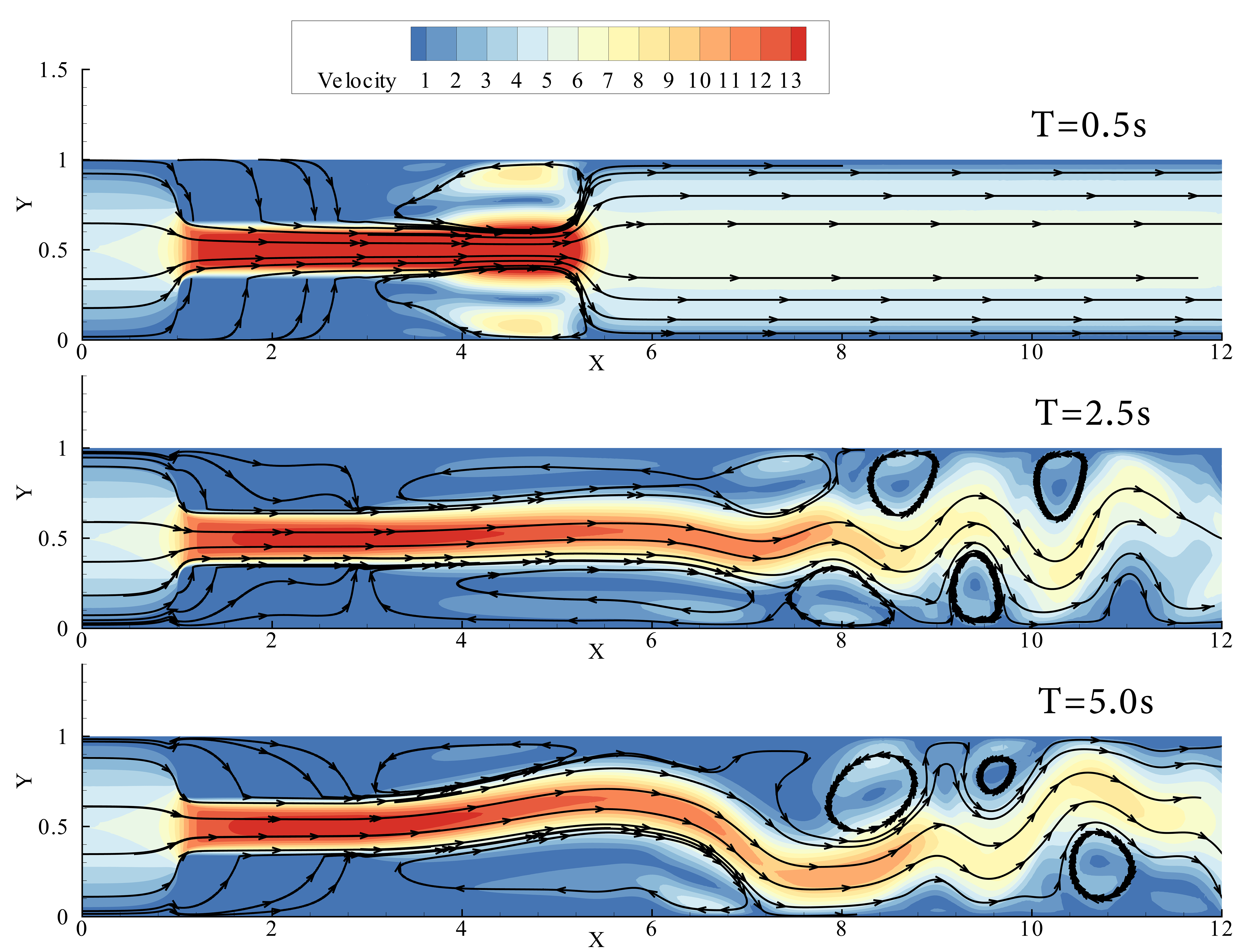}
    \caption{Superimposed streamlines and velocity magnitudes for fluid flow in the channel and Darcy flows through the poroelastic obstacles at $T=0.5$ s (top), $T=2.5$ s (middle) and $T=5.0$ s (bottom). The velocity is in the unit of $cm/s$.}
    \label{fig:streamline_2d}
\end{figure}

Figure \ref{fig:streamline_2d} shows the superimposed streamlines and velocity magnitudes for both fluids flow and Darcy flow in poroelastic obstacles at three distinct time moments. Streamlines remain consistent across the interface between fluids and poroelastic regions, demonstrating a reasonable behavior. In the latter half of the channel, new vortices unremittingly form and break as time advances, illustrating the chaotic and oscillatory nature of the fluid flow. Regarding the fluid pressure and Darcy pressure throughout the computational domain, we note that pressure remains continuous across the interface, making it hard to distinguish poroelastic regions from fluid region, as illustrated in Figure \ref{fig:pressure_2d}. Additionally, Figure \ref{fig:displacement_2d} shows the displacement magnitudes of the two Biot obstacles at selected time steps, indicating that deformation is most pronounced in the lower and upper left corners due to the influence of the fluids. The deformation of poroelasticity domain reaches its peak at $t = 0.5$ s and then gradually stablizes from $t = 2.5$ s to $t = 5.0$ s. Note that a scale factor of 10 is applied for visualization purposes.
\begin{figure}[h!]
    \centering
    \includegraphics[width=1.0\linewidth]{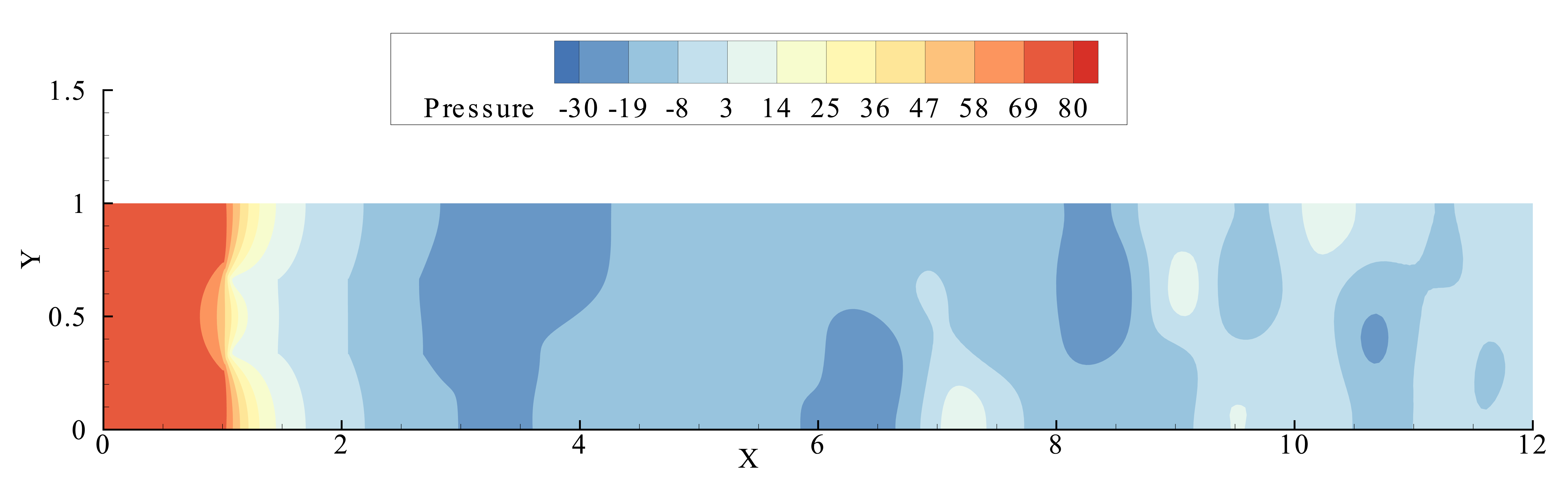}
    \caption{Fluid pressure and Darcy pressure at final time $T=5.0$ s. The pressure is in units of $dynes/cm^2$.}
    \label{fig:pressure_2d}
\end{figure}
\begin{figure}[h]
    \centering
    \includegraphics[width=1\linewidth]{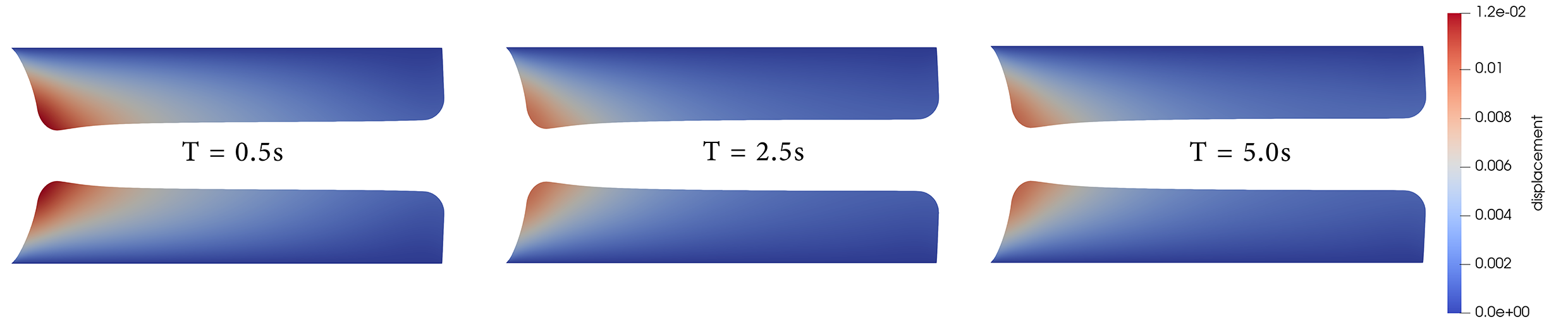}
    \caption{Displacement magnitudes for two poroelastic obstacles, scaled by a factor of 10, at time $T=0.5$ s (left), $T=2.5$ s (middle), and $T=5.0$ s (right). The displacement is in units of $cm$.}
    \label{fig:displacement_2d}
\end{figure}

\subsection{3D simulation of the blood flow through a microfluidic chip with cylindrical poroelastic obstacles} 
To illustrate the stability and effectiveness of our numerical approach, this section presents a $3D$ analysis of flow dynamics within a microfluidic chip, depicted in Figure \ref{fig:domain_3d}. The chip measures 2 cm $\times$ 4.2 cm $\times$ 0.7 cm, with specific dimensional details provided in the figure. The device includes a single inlet on the left for blood entry and dual outlets on the right for discharge.There is a singular inlet on the left for water ingress, and dual outlets on the right for water egress. Within the chip, 10 pillars are strategically placed, all made from poroelastic materials (hydrogels or polymer matrix composites) of varying sizes—six pillars are 0.2 cm in diameter and four are 0.15 cm in diameter. Both the top and bottom ends of these pillars are attached to the chip walls, enforcing $\boldsymbol{\hat{\eta}}=\boldsymbol{\hat{\xi}}=\boldsymbol{0}$ and $\hat{\nabla}\hat{\phi}\cdot\boldsymbol{\hat{n}}_p=0$ for no-flux conditions. The physical parameters remain the same as those used in the $2D$ problem in Section \ref{2dexample}. We set the time step $\Delta t = 0.001$ s, and continued the simulation until a steady state was achieved at $T = 2.0$ s.

Figure \ref{fig:domain_3d} shows that the geometry is discretized using a non-conformal mesh, meaning the fluid and structure domain meshes are not matched at the interface. Our scheme solves both domains simultaneously at each time step, making it crucial that both processes finish in similar CPU times. This allows for localized grid refinement in just one domain while maintaining a balanced workload for both, preventing delays from one domain requiring significantly more processing time. By distributing the workload evenly, the scheme fully leverages parallelization, ensuring efficient use of computational resources and maximizing speedup.
\begin{figure}[htb]
    \centering
\includegraphics[width=1.0\linewidth]{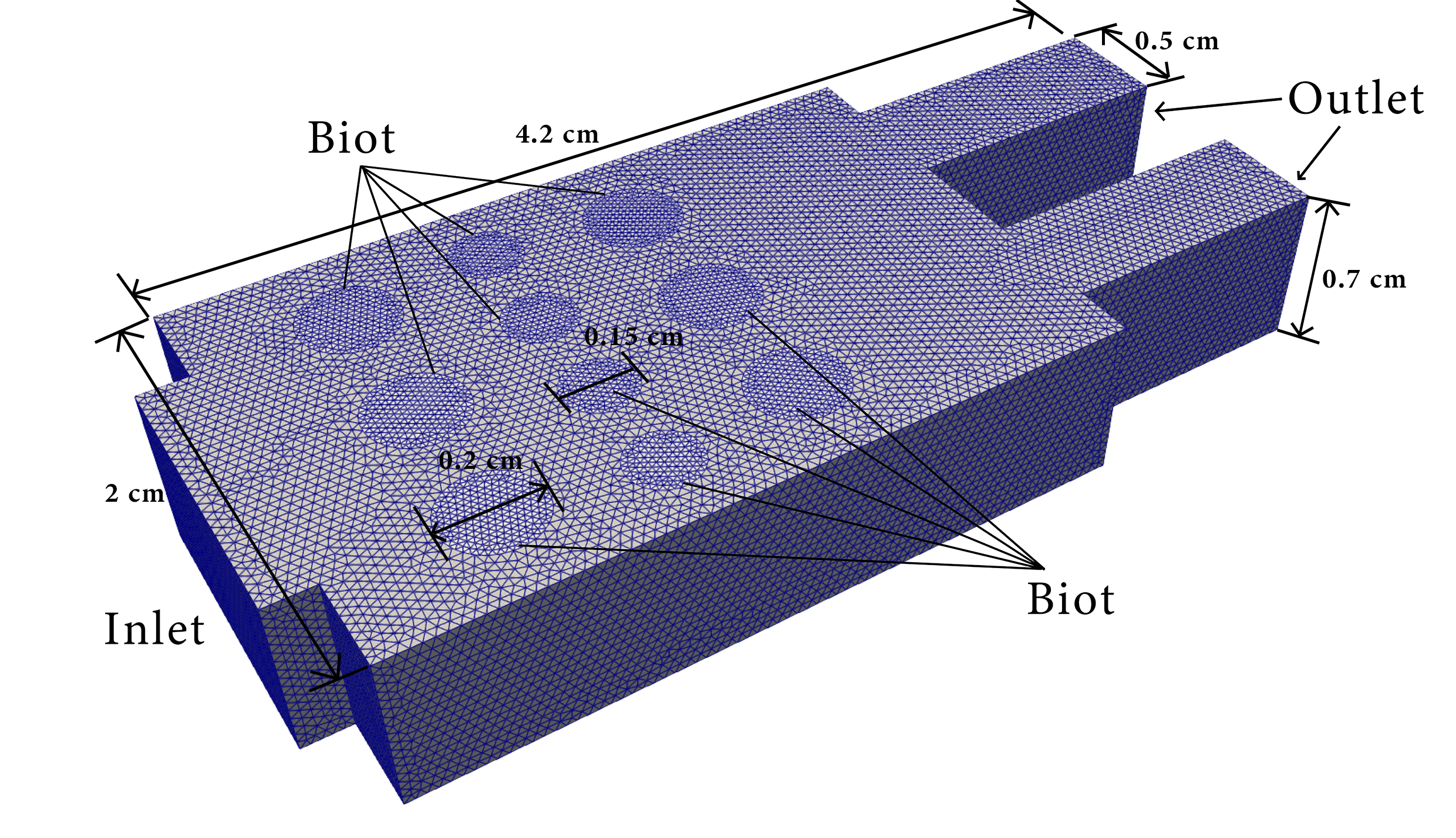}
\caption{A representation of the $3D$ computational domain is provided, where the blood flow region and poroelastic pillars are discretized using non-conformal tetrahedral elements. This results in 275,366 elements for the fluid domain and 156,656 elements for the Biot pillars.}
\label{fig:domain_3d}
\end{figure}

Figure \ref{fig:streamline_3d1} depicts the velocity streamlines of blood flow as it moves around the poroelastic pillars and through the channels, with color contours indicating the fluid velocity magnitude. Higher velocities (red/orange) are concentrated near the fluid inlet and in less obstructed areas, while lower velocities (green/blue) are seen near the pillar surfaces due to increased flow resistance. Vortices form behind each pillar and near the two exits. The displacement of the pillars is also shown, illustrating their deformation under the influence of blood flow. The greatest displacements, reaching approximately $0.0046$ cm, occur around the centers of the smaller pillars, where the fluid velocity reaches the highest. Figure \ref{fig:pressure_3d} illustrates the vorticity magnitude and pressure distribution within the microfluidic chip. In the vorticity subplot, the red regions indicate areas of intense rotational flow, while the blue regions represent areas of lower vorticity. We observed that the areas around the circular pillars show higher vorticity, especially near the edges where flow shear force interacts with pillar surfaces. The bottom plot displays the pressure distribution, in which we observed that areas at the upstream of each pillar experience higher pressure (in red), while the pressure decreases near the downstream. Moreover, the pressure is continuous across the fluid and structure region which satisfies our coupling condition. 
\begin{figure}[htb]
    \centering
    \includegraphics[width=1.0\linewidth]{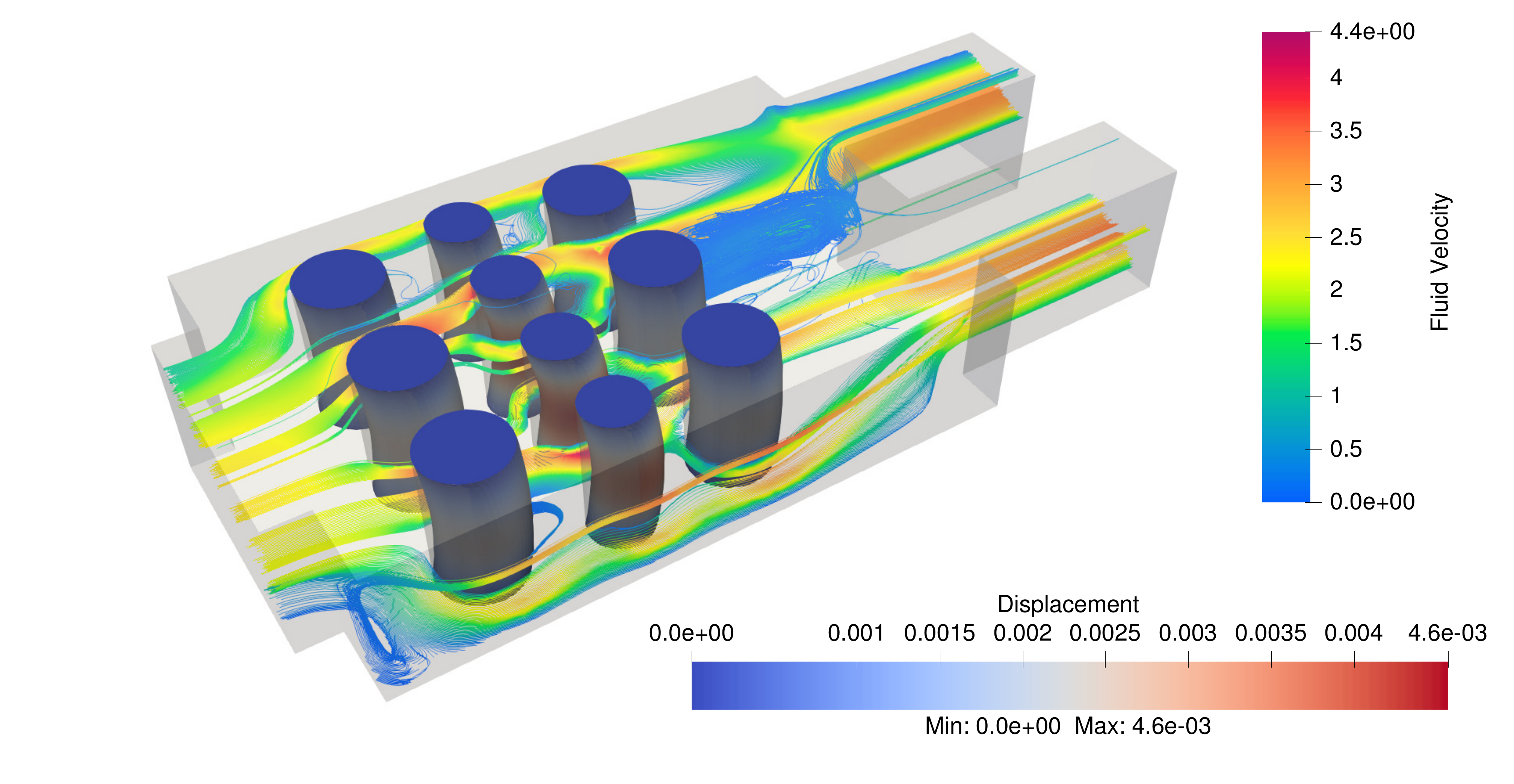}
    \caption{Visualization of fluid flow streamlines and pillar displacements in a $3D$ microfluidic chip at the steady state $T=2.0$ s. The velocity is measured in units of $cm/s$, while the displacement is in units of $cm$. For visualization purposes, the displacement has been magnified by a factor of 10.}
    \label{fig:streamline_3d1}
\end{figure}
\begin{figure}[htb]
    \centering
    \includegraphics[width=1.0\linewidth]{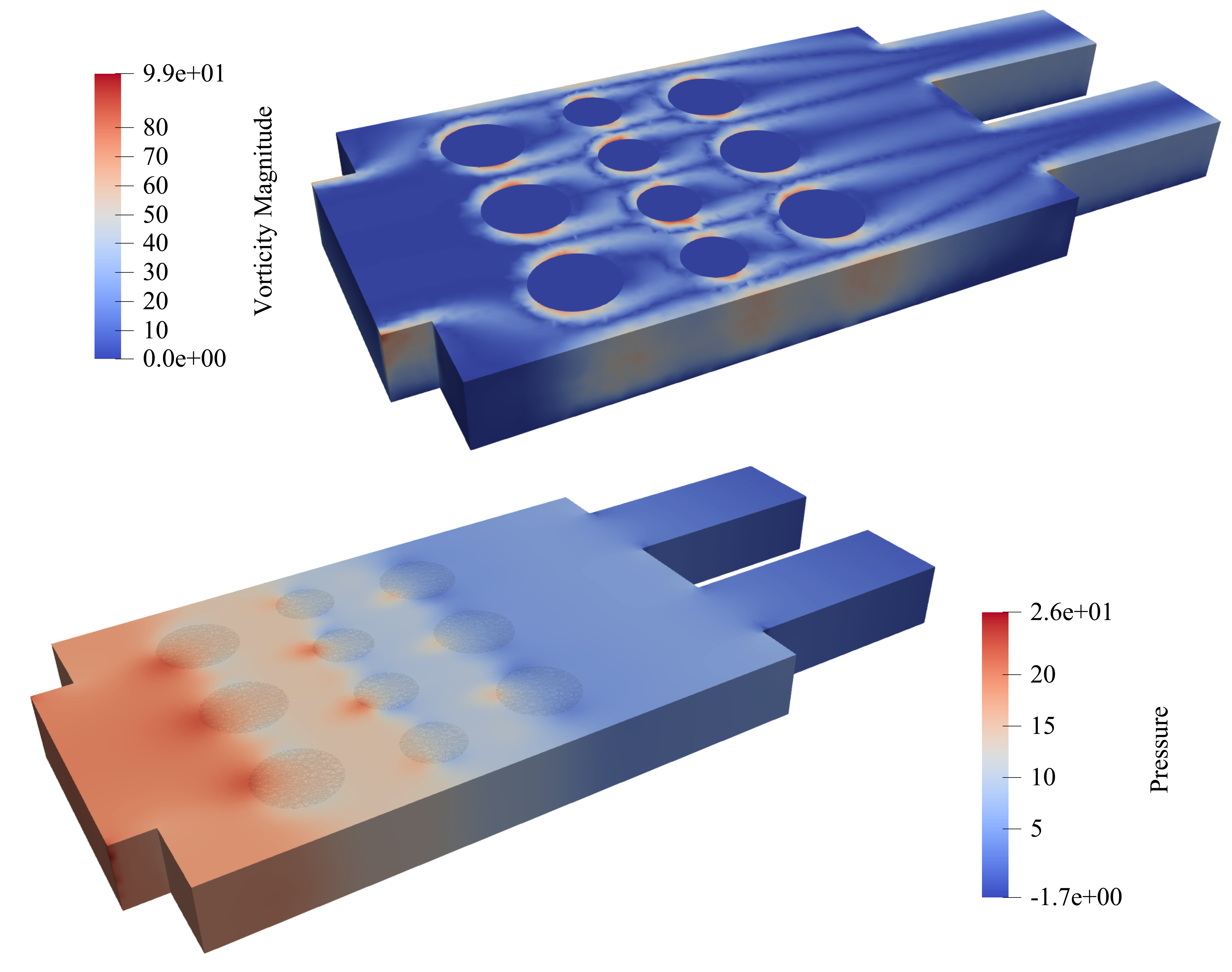}
\caption{Vorticity magnitude (top) and pressure distribution (bottom) for fluid flow around poroelastic pillars at steady state $T=2.0$ s. The vorticity is in the unit of $s^{-1}$ and pressure is in the unit of $dynes/cm^2$. The vorticity plot highlights regions of high rotational flow near the obstacles, with the highest values observed around the obstacle edges. The pressure plot shows elevated pressure near the inflow and obstacle surfaces, with a gradual decrease toward the channel exits.}
\label{fig:pressure_3d}
\end{figure}
\section{Conclusion}
\label{sec:main6}
In this work, we considered a coupled system in a moving domain, where the fluid is governed by the Navier-Stokes equations in the ALE formulation, and the poroelastic material is modeled by the Biot equation. To solve this problem, a fully parallelizable and loosely coupled method was proposed based on Robin-Robin type interface conditions, which were derived from modifying the original interface conditions. Theorem \ref{Equivalency} establishes the equivalence between the original solution and that of the reformulated problem. As outlined in Algorithm \ref{semi}, the solution from the previous time step is utilized to approximate the reformulated interface conditions at the current time step. This allows the coupled NSBiot system to be fully decomposed into separate fluid and Biot subproblems. Since these two subproblems are uncoupled, they can be solved independently without iteration at each time step. 

We investigated the stability of the proposed numerical scheme in the context of a linearized problem, specifically Algorithm \ref{fixedAL}, and demonstrated that it is unconditionally stable. In contrast to some existing methods that either rely on viscoelastic model assumptions or impose restriction on the time step $\Delta t$ for certain physical parameters, our algorithm exhibits robust stability without requiring viscoelasticity. For the parameters in the Robin boundary conditions, we provided reference values based on both numerical considerations and prior work on Stokes-Darcy problems \cite{cao2014parallel}. Additionally, based on solving two benchmark problems with analytic solutions, we observed first-order accuracy in time for this loosely coupled scheme. This approach significantly reduces the computational burden of the original problem by decomposing the coupled system into two subproblems.

To demonstrate the robustness and superb stability property for handling real physical problems in moving domains, we applied Algorithm \ref{semi} to both $2D$ and $3D$ problems involving complex geometries. Notably, for both $2D$ and $3D$ cases, we employed non-conformal meshes, namely the mesh on the interface is not aligned. Since our scheme solve both subproblems at the same time, it is crucial that the two subproblems are handled in comparable CPU times; otherwise, one process may idle while waiting for the other to complete. The use of non-conformal meshes allows for local grid refinement, enabling balanced computational workload among the CPUs. In both $2D$ and $3D$ simulations, we used a relatively large time step size of $\Delta t = 10^{-3}$, comparable to that used by some monolithic solvers, and observed no stability issues. 

In the last, we would also like to acknowledge a limitation of this work: the stability analysis is conducted under the assumption of a linearized problem (in the fixed domain and the fluid problem is governed by the Stokes equation). Nevertheless, the $2D$ and $3D$ examples presented demonstrate the stability and robustness of our numerical scheme for handling the nonlinear fluid poroelastic interaction problem in a moving domain. Future work will focus on an error estimate of the scheme, as well as its potential application to other classical fluid-structure interaction problems and other multiphysics scenarios.

\section*{Acknowledgments}
The third author's research was partially supported by NSF DMS-2247001. The fifth author's research was partially supported by  NSF of China (No.12471406) and the Science and Technology Commission of Shanghai Municipality (Grant Nos. 22JC1400900, 22DZ2229014). 

The authors acknowledge the High-Performance Computing Center (HPCC) at Texas Tech University for providing computational resources (http://www.hpcc.ttu.edu).
\bibliographystyle{plain}
\bibliography{main}
\end{document}